\documentclass[12pt, a4paper]{article} 

\usepackage{amssymb, amsmath, amsfonts, amsthm}
\usepackage[T1, T2A]{fontenc}
\usepackage[utf8]{inputenc}
\usepackage[polish, english]{babel} 
\usepackage{fancyhdr}
\usepackage{latexsym}
\usepackage{mathtools, enumerate, rotating, framed, lipsum}
\usepackage{graphicx}
\usepackage[pass]{geometry}
\usepackage{wrapfig}
\usepackage{rotating}
\usepackage{textcomp}
\usepackage{arcs}
\usepackage{gensymb}
\usepackage{multicol}
\usepackage{titlesec}
\usepackage{chngcntr}
\usepackage{xcolor}
\usepackage{float}
\usepackage{tikz}
\usepackage{courier}
\usepackage{thmtools}
\usepackage{hyperref}
\usepackage{yfonts}
\usepackage{enumitem}
\usepackage{parskip}
\usepackage{etoolbox}
\usepackage{setspace}
\usepackage{pdfpages}
\usepackage{standalone}
\usepackage{svg}
\usepackage{csquotes}
\usepackage[shortcuts]{extdash}

\hypersetup{hidelinks}

\usetikzlibrary{decorations.markings}

\newcommand{\CC}{\mathbb{C}} \newcommand{\RR}{\mathbb{R}} 
\newcommand{\NN}{\mathbb{N}} \newcommand{\ZZ}{\mathbb{Z}}

\newcommand{\abs}[1]{\left| #1 \right|}

\newcommand{\set}[1]{\left\{ #1 \right\}}

\DeclareMathOperator{\Cay}{Cay}

\DeclareMathOperator{\Lk}{Lk}
\DeclareMathOperator{\Res}{Res}

\makeatletter
\def\th@plain{%
  \thm@notefont{}
  \itshape 
}
\def\th@definition{%
  \thm@notefont{}
  \normalfont 
}
\makeatother

\newtheoremstyle{break}
  {\topsep}{\topsep}%
  {\itshape}{}%
  {\bfseries}{}%
  {\newline}{\thmname{#1}\thmnumber{ #2}\thmnote{ (#3)}}%
\theoremstyle{break}
\newtheorem{thm}{Theorem}[section]
\newtheorem{lm}[thm]{Lemma}
\newtheorem{prop}[thm]{Proposition}
\newtheorem{fact}[thm]{Fact}

\newtheorem{oprob}[thm]{Open Problem}
\newtheorem{conj}[thm]{Conjecture}

\newtheoremstyle{defbreak}
  {\topsep}{\topsep}%
  {\upshape}{}%
  {\bfseries}{}%
  {\newline}{\thmname{#1}\thmnumber{ #2}\thmnote{ (#3)}}%
\theoremstyle{defbreak}
\newtheorem{df}[thm]{Definition}
\newtheorem{exam}[thm]{Example}
\newtheorem{rem}[thm]{Remark}

\newcommand{\ol}[1]{\overline{#1}}
\newcommand{\wtl}[1]{\widetilde{#1}}

\newcounter{picturecounter}
\addtocounter{picturecounter}{1}

\newcommand{\com}[1]{}

\title{On the biautomaticity of CAT(0) triangle-square groups} 
\author{Mateusz Kandybo}

\date{November 2024}
\begin{document}

\maketitle
\begin{abstract}
    Following the research from the paper ``Triangles, squares and geodesics'' of Rena Levitt and Jon McCammond we investigate the properties of groups acting on CAT(0) triangle-square complexes, focusing mostly on biautomaticity of such groups. In particular we show two examples of nonpositively curved triangle-square complexes $X_1$ and $X_2$, such that their universal covers violate conjectures given in the aforementioned paper. This shows that the Gersten-Short geodesics cannot be used as a way of proving biautomaticity of groups acting on such complexes. Lastly we give a proof of biautomaticity of $\pi_1(X_1)$, however the biautomaticity of $\pi_1(X_2)$ remains unknown.
\end{abstract}

\section{Introduction}

Groups acting on nonpositively curved spaces are one of the leading research topics in modern geometric group theory. Out of many types of such spaces two families are particularly intriguing: the CAT(0) cubical complexes, introduced by M. Gromov in \cite{gromov_hyperbolic}, and systolic spaces, developed independently by many authors (\cite{bridged_original}, \cite{bridged_second}, \cite{haglund2003}, \cite{systolic}). Both of those of classes of spaces admit a number of geometric and combinatorial properties which makes them a noteworthy subject of research. 

On the other hand, two important properties of groups, established by several authors in \cite{word_processing}, are automaticity and biautomaticity. It was shown that several natural classes of groups acting on nonpositively curved spaces, including so-called word hyperbolic groups, are either automatic or biautomatic. The analysis of properties of groups acting geometrically and cellularly on CAT(0) cubical complexes (conducted by G. A. Niblo and L. Reeves \cite{niblo_reeves}) and systolic spaces (conducted by T. Januszkiewicz and J. Świątkowski in \cite{systolic}) reveal that all such groups are biautomatic as well. That raises the question: are there any generalizations of both of those types of spaces such that groups acting on them are also biautomatic?

The foregoing question is the main point of a research done by R. Levitt and J. McCammond in their paper \cite{triangles_squares}. The authors attempted to prove biautomaticity for groups acting geometrically and cellularly on CAT(0) triangle-square complexes. This topic is particularly interesting as the existence of not biautomatic $2$-dimensional CAT(0) groups is a long standing open problem, stated as question $43$ in \cite{farb}. In $3$ dimensional case analogous groups do exist, which was shown in \cite{leary} and \cite{hughes-valiunas}. 

The results given in this paper revolve mostly around investigating whether so-called Gersten-Short geodesics are fellow-traveling inside the triangle-square complexes. Authors has shown that such geodesics fellow travel if a triangle-square complex is $\delta$-hyperbolic, however in case a complex contains isometrically embedded flats this property might not be true. In particular they formulate two conjectures regarding existence of so-called radial and thoroughly crumpled flats inside the universal cover of any compact locally CAT(0) complex $X$.

The main result of this paper is showing direct counterexamples to conjectures of Levitt and McCammond.

\textbf{Main Theorem}\\
\textit{There exist a compact locally CAT(0) triangle-square complex $X$ having both a radial flat and an infinite family of distinct thoroughly crumpled flats embedded into its universal cover.}

We give two examples of spaces, each of them violating both of aforementioned conjectures. The first one of them is homeomorphic to a locally CAT(0) simplicial complex, which is enough to show that its fundamental group is biautomatic, however the biautomaticity of the other one is unclear. We also show that the Gersten-Short geodesics do not fellow travel in radial flats embedded into the universal covers, which rule out possibility of using such geodesics for proving biautomaticity in general case.

This paper is split into $6$ sections. In section \ref{preliminaries} the essential definitions, notions and lemmas are stated. The proofs in this section are usually omitted and replaced with references to the relevant sources in literature. In section \ref{flats_in_complexes} an important lemma about embedding flats is proven. Furthermore, in that section definitions of flats that will be relevant for proving that our examples violates the conjectures of Levitt and McCammond are stated. The point of the next two sections -- \ref{first_example} and \ref{second_example} -- is presenting our counterexamples together with proofs, pictures and important remarks. In section \ref{final remarks}, conclusions are drawn from proven results and new conjectures for groups acting on triangle-square complexes are stated. Moreover, in that section some final thoughts are shared. 

\textbf{Acknowledgement}\\
This paper arose as a natural continuation of the author's master's thesis. The author would like to thank his supervisor, Motiejus Valiunas, for his ideas, suggestions and help with the general shape of this paper. The author is also grateful to Jacek Świątkowski and Damian Osajda for insight regarding some literature and proofs. The work presented here was partially supported by the Carlsberg Foundation, grant CF23-1226.
 
\section{Preliminaries}\label{preliminaries}

\subsection{Geometry of Complexes}

In this subsection we give definitions of polyhedral complexes. The main sources of those definitions are a paper \cite{triangles_squares} and a book \cite{bh}, however other sources were also used (\cite{systolic}, \cite{spanier1989algebraic}).

One of the main objects discussed in this paper are various two dimensional Eulidean polyhedral complexes. Below we give some basic terminology and a definition of such complexes. Note that our definition of polyhedral complexes differ slightly from the standard one presented in \cite{bh}, as we allow two faces of the same cell to be identified. The usual polyhedral complexes can be obtained from polyhedral complexes in the sense of this paper, for example by baricentrically subdividing each cell, hence most of the facts presented in the book of M. Bridson and A. Haefliger can still be applied in our setting. For more details see chapter I.7 of \cite{bh} and section 2 of \cite{triangles_squares}.

\begin{df}[polyhedral cell]
    A convex hull of any finite number of points in Euclidean space $\RR^n$ is called an (Euclidean) \textit{polyhedral cell}. A \textit{dimension} of a cell $C$ is the dimension of the smallest affine subspace of $\RR^n$ containing $C$. The cells of dimension $0, 1$ and $2$ are respectively called \textit{vertices}, \textit{edges} and \textit{polygons}. If $C$ lies in a closed halfspace bounded by affine plane $H$ and the intersection $C\cap H$ is nonempty then such intersection is called a \textit{face} of $C$. Moreover if a face is not equal to $C$, then it is called a \textit{proper face}. Any terminology described above applies to faces as they are also polyhedral cells.
\end{df}

\begin{df}[polyhedral complex]
    A \textit{polyhedral complex} is a topological space obtained by gluing a family of polyhedral cells via isometric identifications of their faces. For practical reasons we assume that any face of a cell inside a polyhedral complex is a cell itself. It turns out that if the aforementioned family has only finitely many isometry types of cells, then the obtained polyhedral complex is in fact a metric space (see theorem I.7.13 in \cite{bh}). A continuous map $f\colon X\to Y$ between two polyhedral complexes is called a \textit{cellular map} if it is a local isometry on each cell of $X$ and the image of any cell is also a cell. 
\end{df}

\begin{df}[$k$-skeleton]
    A \textit{$k$-skeleton} of a polyhedral complex $X$ is a subcomplex of $X$ containing only all at most $k$ dimensional cells of $X$. The $k$-skeleton of $X$ is denoted by $X^{(k)}$ and the metric on $X^{(k)}$ is obtained as a natural metric on the complex $X^{(k)}$ itself rather than a restriction of metric from $X$.
\end{df}

Below we present an important fact regarding covers of polyhedral complexes. Proof of this fact can be found in chapter 3 of \cite{spanier1989algebraic}.

\begin{rem}[covers]
    One of the important example of cellular maps is the covering map. Let $X$ be a polyhedral complex with finitely many isometry types of cells, let $Y$ be a cover of $X$ together with a covering map $p: X\to Y$. Then $Y$ admits a structure of a polyhedral complex with the same isometry types of cells as $X$ such that the map $p$ is cellular.
\end{rem}

Below we present $3$ examples of polyhedral complexes, that were motivating for research in this paper. The first two of them are described in detail in chapter I.7 of \cite{bh} and the last one is presented in section 2 of \cite{triangles_squares}.

\begin{exam}[cubical complex]
    A \textit{cubical complex} is a polyhedral complex obtained by gluing cubes.
\end{exam}

\begin{exam}[simplicial complex] 
    An \textit{abstract simplicial complex} is defined in an abstract way as a pair of the set $V$ called the set of vertices and the set $\mathcal{S}\subseteq \mathcal{P}(V)$ called the set of simplices satisfying the following properties:
    \begin{itemize}
        \item all sets $\sigma\in \mathcal{S}$ are finite and nonempty;
        \item $\set{v} \in \mathcal{S}$ for every $v\in V$;
        \item if $\sigma\in \mathcal{S}$ and $\emptyset\neq \sigma'\subseteq \sigma$, then $\sigma'\in \mathcal{S}$.
    \end{itemize}
    With each abstract simplicial complex we can associate a polyhedral complex obtained by replacing each of the sets $\sigma\in\mathcal{S}$ by a $\abs{\sigma}-1$ dimensional simplex with edges of length $1$ and gluing each two simplices obtained from $\sigma$ and $\sigma'$ if they have a common face $\sigma\cap \sigma'$ via that common face. Such complex is called a \textit{simplicial complex}.
\end{exam}

\begin{exam}[triangle-square complex]
    A \textit{triangle-square complex} is a polyhedral complex obtained by gluing equilateral triangles with sides of length $1$ and squares with sides of length $1$. Through this paper we will use the names \textit{unit triangle} and \textit{unit square} for describing triangles and squares as above. Moreover we will also use the name \textit{unit hexagon} for describing regular hexagons with all sides of length $1$.
\end{exam}

The following remark gives an example of triangle-square complex, that will be particularly important for the purpose of making constructions later in this paper. This example is also described in section 8 of \cite{triangles_squares}.

\begin{exam}[Eisenstein plane]
    An example of a triangle-square complex is the plane tiled with unit triangles. Note that vertices on such plane can be identified with the Eisenstein integers $\ZZ\left[\frac{-1+\sqrt{3}i}{2}\right]$ thus we call such plane an \textit{Eisenstein plane} and denote it as $\mathcal{E}$. Moreover we also use the identification of this plane with $\CC$ for the purpose of defining objects on $\mathcal{E}$.
\end{exam}

Since the main object of this paper is studying concrete examples of triangle-square complexes then there is demand to create a notations to directly construct and study such complexes. Our approach of displaying triangle-square complexes is presented in the remark below.

\begin{rem}[constructing complexes]\label{practical_gluing}
    To define a two dimensional polyhedral complex we need to define the gluing between the cells. For that purpose we use the notation $e:AB \sim YZ$ meaning that an oriented edge $AB$ of a cell is glued isometrically to the oriented edge $YZ$ of a (possibly the same) cell in a way that preserve orientation and the edge in a complex obtained in that way is labeled as $e$. In the illustrations of the polyhedral cells we denote the fact that two oriented edges are glued together by labeling them in the same way. Moreover we also use colors and different types arrows to further help distinguish different edges in a final complex.
    
    An example of a complex built from a unit square $ABCD$ and a rhombus $EFGH$ such that $\sphericalangle EFG = \frac{\pi}{3}$ with sides of unit length with gluing:
    \begin{align*}
        e:AB \sim DC \sim EF \sim HG, && f:BC \sim DA, && g:GF \sim HE
    \end{align*}  
    is presented in the figure \ref{complex_example}.

    \begin{figure}[!ht]
    \begin{center}
\begin{tikzpicture}[scale=0.9]
\fill[fill=yellow!20] (-1,1)--(1,1)--(1,3)--(-1,3);
\draw[black!50!green, thick, ->] (-1,1)--(0,1);
\draw[black!50!green, thick] (1,1)--(0,1);
\filldraw[black!50!green] (0,1) circle (0pt) node[anchor=north]{$e$};
\draw[black!50!green, thick, ->] (-1,3)--(0,3);
\draw[black!50!green, thick] (1,3)--(0,3);
\filldraw[black!50!green] (0,3) circle (0pt) node[anchor=south]{$e$};
\draw[blue, thick, ->] (-1,3)--(-1,2);
\draw[blue, thick] (-1,2)--(-1,1);
\filldraw[blue] (-1,2) circle (0pt) node[anchor=east]{$f$};
\draw[blue, thick] (1,3)--(1,2);
\draw[blue, thick, ->] (1,1)--(1,2);
\filldraw[blue] (1,2) circle (0pt) node[anchor=west]{$f$};
\filldraw[black] (-1,1) circle (2pt) node[anchor=north east]{$A$};
\filldraw[black] (1,1) circle (2pt) node[anchor=north west]{$B$};
\filldraw[black] (1,3) circle (2pt) node[anchor=south west]{$C$};
\filldraw[black] (-1,3) circle (2pt) node[anchor=south east]{$D$};

\fill[fill=red!15] (0,-3)--(1.732,-2)--(0,-1)--(-1.732,-2);
\draw[black!50!green, thick, ->] (0,-3)--(0.866,-2.5);
\draw[black!50!green, thick] (1.732,-2)--(0.866,-2.5);
\filldraw[black!50!green] (0.866,-2.5) circle (0pt) node[anchor=north west]{$e$};
\draw[blue, thick, ->>] (0,-1)--(0.866,-1.5);
\draw[blue, thick] (1.732,-2)--(0.866,-1.5);
\filldraw[blue] (0.866,-1.5) circle (0pt) node[anchor=south west]{$g$};
\draw[blue, thick, ->>] (-1.732,-2)--(-0.866,-2.5);
\draw[blue, thick] (0,-3)--(-0.866,-2.5);
\filldraw[blue] (-0.866,-2.5) circle (0pt) node[anchor=north east]{$g$};
\draw[black!50!green, thick, ->] (-1.732,-2)--(-0.866,-1.5);
\draw[black!50!green, thick] (0,-1)--(-0.866,-1.5);
\filldraw[black!50!green] (-0.866,-1.5) circle (0pt) node[anchor=south east]{$e$};
\draw[black!50, thick, dashed] (0,-1)--(0,-3);
\filldraw[black] (0,-3) circle (2pt) node[anchor=north]{$E$};
\filldraw[black] (1.732,-2) circle (2pt) node[anchor=west]{$F$};
\filldraw[black] (0,-1) circle (2pt) node[anchor=south]{$G$};
\filldraw[black] (-1.732,-2) circle (2pt) node[anchor=east]{$H$};
\end{tikzpicture}\hspace{5cm}
\begin{tikzpicture}[scale=1.1]
    \draw[black, thick] (-2,0)--(-1,-2.45)--(1,-2.45)--(2,0)--(1,2.45)--(-1,2.45)--(-2,0);
    \draw[red, thick, dashdotted] (-2,0)--(0,-1.225)--(2,0)--(0,1.225)--(-2,0);
    \filldraw[black] (-2,0) circle (1.5pt) node[anchor=east]{$e$};
    \filldraw[black] (-1,-2.45) circle (1.5pt);
    \filldraw[white] (-1,-2.45) circle (1pt);
    \filldraw[black] (1,-2.45) circle (1.5pt) node[anchor=north west]{$\ol{g}$};
    \filldraw[black] (2,0) circle (1.5pt) node[anchor=west]{$\ol{e}$};
    \filldraw[black] (1,2.45) circle (1.5pt);
    \filldraw[white] (1,2.45) circle (1pt);
    \filldraw[black] (-1,2.45) circle (1.5pt) node[anchor=south east]{$g$};
    \filldraw[black] (0,-1.225) circle (1.5pt) node[anchor=north]{$f$};
    \filldraw[black] (0,1.225) circle (1.5pt) node[anchor=south]{$\ol{f}$};
\end{tikzpicture}
\end{center}
    \caption{Left: a triangle-square complex, right: a link of the unique vertex in such complex}
    \label{complex_example}
\end{figure}
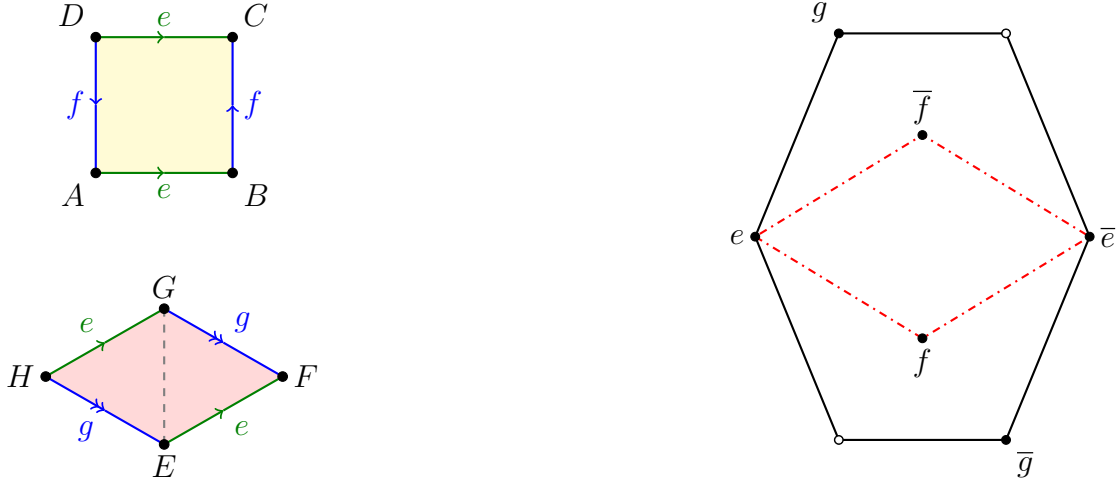
    
    In some cases we also allow the gluing of diagonals of cells instead of only edges. The spaces obtained in such a way are also polyhedral complexes with smaller cells, however it allows us to define such complexes in a more manageable way. Finally, for the purpose of creating a triangle-square complex we can subdivide a cells in our complex into unit triangles and squares. In that way a complex described above is a triangle-square complex because we can divide a rhombus $EFGH$ along the diagonal $EG$.
\end{rem}

In order to work with geometry of complexes we need to analyze the way in which the cells are glued together. A link is a very important and useful tool for analyzing complexes in a combinatorial way. Below we present a preliminary definition of the length of a curve from chapter I.1 of \cite{bh} as well as the definition of a link that is described in more detail in chapter I.7 of \cite{bh} and section 2 of \cite{triangles_squares}.

\begin{df}[length of a curve] 
    Let $X$ be a polyhedral complex. We can define the \textit{length} of any curve $\eta\colon [0,1]\to X$ as
    $$l(\eta) = \lim_{n\to\infty}\sup_{0=t_0\leqslant\ldots\leqslant t_n=1}\sum_{i=0}^{n-1}d(\eta(t_i),\eta(t_{i+1})).$$
\end{df}

\begin{df}[link]\label{def:link}
    Let $v$ be a vertex of $X$. Moreover, let $S(v, \varepsilon)$ be the sphere around $v$ in $X$ of radius $\varepsilon$ equipped with the metric 
    $$d_{v,\varepsilon}(x,x')=\frac{1}{\varepsilon}\inf\set{l(\eta)|\eta\colon[0,1]\to S(v,\varepsilon), \eta(0)=x, \eta(1)=x'}.$$
    If $X$ has only finitely many isometry types of cells there is sufficiently small $\varepsilon$ such that for every positive $\varepsilon' \leqslant \varepsilon$ the spaces $S(v,\varepsilon)$ and $S(v, \varepsilon')$ are isometric. We say that $S(v, \varepsilon)$ as above is a \textit{link} of $v$ in $X$ and denote it as $\Lk(v,X)$.
\end{df}

Note that in case of two dimensional complexes links are graphs. The vertices in such graphs correspond to either initial or final parts of oriented edges in complex and edges in such graphs correspond to angles in two dimensional cells. Moreover the length of any edge in a graph is exactly the measure of a corresponding angle in a polygon. That observation makes analyzing links in triangle-square complexes approachable and allows us to establish the following notation.

\begin{rem}[drawing links]
    We use a filled black dot to indicate a vertex in link corresponding to either initial or final part of a labeled oriented edge. If such dot corresponds to an initial part of an edge $e$, then it is itself labeled $e$; if such dot corresponds to a final part of an edge $e$, then it is labeled $\ol{e}$. Moreover we use hollow black dots to indicate a vertex in link corresponding to an unlabeled edge. We also use a solid black edges to indicate an edge of length $\frac{\pi}{3}$ and dash-dotted red edges to indicate an edge of length $\frac{\pi}{2}$. The link of the only vertex of complex described in remark \ref{practical_gluing} is shown in the figure \ref{complex_example}.
\end{rem}

 There is an another, combinatorial approach to defining links in simplicial complexes. We will use this definition while defining systolic complexes. For more detail and context regarding combinatorial links see chapter I.7 of \cite{bh} and section 1 of \cite{systolic}.

\begin{df}[combinatorial link]\label{comb_lk}
   We can see the link of a vertex $v$ inside simplicial complex $X$ as a subcomplex with simplices $\tau$ such that $\set{v} \cap \tau = \emptyset$ and $\tau\cup \set{v}$ is a simplex in $X$. Such subcomplex is homeomorphic to the standard link defined in the definition \ref{def:link}. In particular if we set the length of every edge to be $\frac{\pi}{3}$ in a link of a vertex in $2$ dimensional complex, then such complex will be isometric to the standard link. This also allows us to generalize the definition of the link for any simplex $\sigma$: the \textit{link} of $\sigma$ in $X$ is the subcomplex consisting of simplices $\tau$ such that $\sigma \cap \tau = \emptyset$ and $\tau\cup \sigma$ is a simplex in $X$. We denote links in such abstract sense as $X_\sigma$.    
\end{df}

The next two theorems provide a valuable methods for investigating polyhedral complexes. First of them, the Gromov link condition, is stated in two dimensional case as the complexes of our interest are two dimensional. The latter one however is stated in full generality. Those theorems are described in detail and proven in chapters II.4 and II.5 of \cite{bh}.

\begin{thm}[Gromov link condition]\label{GLC}
    Let $X$ be a polyhedral complex in which each cell is at most $2$-dimensional and with finitely many isometry types of cells. $X$ is locally CAT(0) if and only if for each vertex $v\in X$ every injective loop in $\Lk(v, X)$ has a length at least $2\pi$.
\end{thm}

\begin{thm}[Cartan-Hadamard theorem]\label{CH} 
    Let $X$ be a complete connected metric space. If $X$ is locally CAT(0) then the universal cover $\wtl{X}$ is CAT(0).
\end{thm}

The CAT(0) spaces are an important example of nonpositively curved spaces. The theory of those spaces is well-developed and we know a lot of properties of such spaces. The following two properties, proof of which be found in chapter II.1 of \cite{bh}, will be important for this paper. 

\begin{fact}\label{simply_connected+local_geodesics}
    All CAT(0) spaces are connected and simply connected. Moreover local geodesics in CAT(0) spaces are geodesics.
\end{fact}

The definition below describes what is considered a ``nice'' group action in geometric group theory. Some properties of that action are described in section III.$\Gamma$ of \cite{bh}, however such action isn't named in that book.

\begin{df}[geometric action] 
    We say that $G$ acts on $X$ \textit{by isometries} if for any $g\in G$ the map $x \mapsto g\cdot x$ is a isometry. Moreover we say that the action of $G$ on $X$ is \textit{proper} if for every compact $K \subseteq X$ the set $\set{g\in G | g\cdot K\cap K \neq \emptyset}$ is finite. Lastly we say that the action of $G$ on $X$ is \textit{cocompact} if the space $X/G$ is compact. If $G$ acts on $X$ by isometries, properly and cocompactly, then we say that such action is \textit{geometric}. 
\end{df}

\begin{rem}
    Many classical examples of group actions turns out to be geometric. One of such examples is an action of a finitely generated group on its Cayley graph, presented in the definition \ref{Cayley graph} (see I.8.11 in \cite{bh}). Another one is an action of a fundamental group of a compact finite-dimensional polyhedral complex on its universal cover via deck transformations (see ``other geometric models of groups'', chapter 7 in \cite{dructu2018geometric}). The latter one is particularly important for research in this paper as we will construct our counterexamples by taking a fundamental group of a triangle-square complex acting on its universal cover.
\end{rem}

The systolic complexes were discovered independently by several researchers (\cite{haglund2003}, \cite{systolic} as a way of defining nonpositive curvature for simplicial complexes and \cite{bridged_original}, \cite{bridged_second} in graph theoretical setting under the name of bridged graphs). Below we present the definition of systolic complex stated in \cite{systolic}. 

\begin{df}[systolicity]
    A \textit{cycle} in a simplicial complex $X$ is a subcomplex of $X$ isomorphic to a triangulation of $S^1$. A subcomplex $X'$ of $X$ is \textit{full} if for every collection of vertices in $X'$  which is contained in a common simplex in $X$ is also contained in a common simplex in $X'$. A \textit{residue} of a simplex $\sigma$ in $X$ is the smallest subcomplex of $X$ containing all simplices $\tau$ such that $\sigma\subseteq \tau$. We denote the residue of $\sigma$ in $X$ as $\Res(\sigma,X)$. A \textit{systole} of a complex $X$, denoted $sys(X)$, is the minimal length of a full cycle in $X$. Given a natural number $k\geqslant 4$ we say that a simplicial complex is:
    \begin{itemize}
        \item $k$\textit{-large} if $sys(X) \geqslant k$ and $sys(X_\sigma) \geqslant k$ for each simplex $\sigma$ of $X$;
        \item \textit{locally $k$-large} if residue of every simplex of $X$ is $k$-large;
        \item \textit{$k$-systolic} if it is connected, simply connected and locally $k$-large.
    \end{itemize}
    If a complex is $6$-systolic we call it \textit{systolic}.
\end{df}

Below we present a fact from the paper \cite{systolic}. For the sake of completeness we also provide a proof of this fact as it was omitted in the original source.

\begin{fact}\label{2-dim-systolic} 
    Let $X$ be a simplicial complex with every cell having the dimension at most $2$. Then the following conditions are equivalent:
    \begin{itemize}
        \item[(i)] $X$ is systolic;
        \item[(ii)] $X$ is CAT(0).
    \end{itemize}
\end{fact}

\begin{proof}
    Firstly, we note that systolic spaces are connected and simply connected by the definition and CAT(0) spaces are connected and simply connected by the fact \ref{simply_connected+local_geodesics}. Therefore by using the Cartan-Hadamard theorem it is enough to check that $X$ is locally $6$-large if and only if $X$ is locally CAT(0).

    $\implies$) Let $X$ be a $2$-dimensional locally $6$-large simplicial complex. By the Gromov link condition (theorem \ref{GLC}) it is enough to check that for each vertex $v$ of $X$ the link $\Lk(v,X)$ has no isometric cycles of length less than $2\pi$. By the remark in the definition of combinatorial link (definition \ref{comb_lk}) the link is isometric to an combinatorial link $X_v$ where each edge has length $\frac{\pi}{3}$, so it is enough to show that there are no cycles of length less than $6$ in the combinatorial link $X_v$. However, we have $X_v = \Res(v,X)_v$, so by $6$-largeness of $\Res(v,X)$ we know that $X$ is locally CAT(0).

    $\impliedby$) Let $X$ be a $2$-dimensional locally CAT(0) simplicial complex. We need to show that the residue of any simplex $\sigma$ is locally $6$-large. If $\sigma$ is a $2$-dimensional simplex, then the residue of $\sigma$ is a subcomplex containing only $\sigma$ and thus neither the residue nor any combinatorial link inside the residue have any full cycles inside. If $\sigma$ is a $1$-dimensional simplex, then the residue of $\sigma$ is the smallest subcomplex containing all of the cells containing $\sigma$ so it is isomorphic to some number of triangles with one of their edges identified. It is easy to see that in that case neither the residue nor each of the combinatorial links inside in the residue have any full cycles inside. If $\sigma$ is a $0$-dimensional simplex however the residue $Res(\sigma, X)$ can have some full cycles, but none of those cycles can contain $\sigma$, thus it is enough to check that $Res(\sigma, X)_\sigma$ has no full cycles of length less than $6$. We know however that $Res(\sigma, X)_\sigma = X_\sigma$ for a $0$-dimensional simplex $\sigma$ and by the remark in the definition of the combinatorial link (definition \ref{comb_lk}) it is enough to say that there are no isometric cycles of length less than $2\pi$ in the link $\Lk(v,X)$, which follows from the fact that $X$ is locally CAT(0). Moreover, if we consider any simplex $\tau \neq \sigma$ then the combinatorial link $Res(\sigma, X)_\tau$ is either empty (when $\tau$ is $2$-dimensional), a graph consisting of some number of vertices without any edges (when $\tau$ is $1$-dimensional) or it is a graph consisting of some number of vertices, one of which is $\sigma$ and edges of form $\set{v,\sigma}$ (when $\tau$ is $0$-dimensional). In all of those cases there are no full cycles inside $Res(\sigma, X)_\tau$, which ends the proof.
\end{proof}

\subsection{Biautomaticity}

In this subsection we state the most important definitions and facts in the theory of biautomatic groups. This theory was primarily developed in book \cite{word_processing}, but the methods used in this paper come mostly from newer sources \cite{path_systems} and \cite{triangles_squares}.

Below we present the definition of a path similar to the one from \cite{path_systems}. Note that this definition isn't standard and a path in our sense is called a walk by most authors, however we decided to use it as a way of being consistent with the paper \cite{path_systems}.

\begin{df}[path]
    Let $X$ be a graph (possibly with loops and multiple edges joining the same pair of vertices). We say that a \textit{path} is a finite sequence of oriented edges $e_1,e_2,\ldots,e_n$ in $X$ such that the final vertex of an edge $e_i$ and the initial vertex of an edge $e_{i+1}$ are the same for all $1<i<n$. The \textit{length} of a path $\gamma$, denoted $\text{len}(\gamma)$, is the number of edges in that path. We say that a vertex is an \textit{$i$-th vertex} of a path $e_1,e_2,\ldots,e_n$ if it is an initial vertex of the edge $e_{i+1}$. For convenience we say that the $i$-th vertex of a path of length $n$ where $i \geqslant n$ is the terminal vertex of the edge $e_n$. We denote the $i$-th vertex of a path $\gamma$ as $\gamma^{(v)}_i$. We say a path $\gamma$ of length $n$ \textit{starts} at the vertex $\gamma^{(v)}_0$ and \textit{ends} at the vertex $\gamma^{(v)}_n$. Finally a \textit{path system} in $X$ is any set of paths in $X$.
 \end{df}

Below we present the definition of a nondeterministic finite automaton in a way that is slightly unusual, yet will be easier to generalize into a notion of automaton over $(G,X)$ in later part of the paper. More information about finite automatons can be found in the chapter 2 of the book \cite{hopcroft}. 

\begin{df}[nondeterministic finite automaton] 
    Let $\Sigma$ be a finite set called \textit{alphabet}. Any finite sequence of elements of $\Sigma$ is called a \textit{word} over $\Sigma$ and any set of words is called a \textit{language} over $\Sigma$. A \textit{nondeterministic finite automaton} (\textit{NFA}) $\mathcal{M}$ is a finite directed graph (possibly with loops and multiple edges between the same pair of vertices) with edges labeled by $\Sigma$ and two selected sets of vertices (not necessarily disjoint): $I$ called the \textit{initial states} and $F$ called \textit{final states}. 

    We say that the NFA $\mathcal{M}$ \textit{accepts} word $w$ if there is a path in $\mathcal{M}$ starting at an initial state and ending in final state which is labeled by $w$. The \textit{language of $\mathcal{M}$} is a set of all words accepted by $\mathcal{M}$. Finally a language $L$ is \textit{regular} if there exist a NFA $\mathcal{M}$ such that $L$ is the language of $\mathcal{M}$.
\end{df}

The following definition is a property of paths that is important for describing biautomaticity. We present the standard definition of fellow traveling together with its generalization that is stated in section 1 of \cite{path_systems} and in section 3 of \cite{triangles_squares}.

\begin{df}[fellow traveler property]\label{ftp-def} 
    Let $X$ be a graph. We say that two paths $\gamma$ and $\eta$ are such that \textit{$k$-fellow travel} if the distance between those two paths at any point in time is bounded by $k$ from above, that is $d(\gamma^{(v)}_i,\eta^{(v)}_i) \leqslant k$ for any natural $i\leqslant \max(\text{len}(\gamma),\text{len}(\eta))$. We say that a path system $\mathcal{P}$ satisfies \textit{2-sided synchronous fellow traveler property} if there is a constant $k$ such that any two paths from $\mathcal{P}$ that start and end at distance at most one apart $k$-fellow travel. Through this paper we will abbreviate the name and refer to 2-sided synchronous fellow traveling as fellow traveling.

    The definition above can be generalized to the following condition we call \textit{generalized fellow traveler property}. We say that a path system $\mathcal{P}$ satisfies generalized fellow traveler property if there are constants $C, D$ such that for any $\ell \geqslant 0$ any two paths from $\mathcal{P}$ that start and end at distance at most $\ell$ apart $(C\ell+D)$-fellow travel. This condition will be particularly important for stating the theorem \ref{swiatkowskis_thm}.
\end{df}

Below we present the definition of a Cayley graph, that can be found in section 3 of \cite{triangles_squares}.

\begin{df}[Cayley graph]\label{Cayley graph}
    A \textit{(right) Cayley graph} over $G$ with respect to a generating set $A$ is a oriented, labeled graph with the vertices indexed by an elements of $G$ and oriented edges labeled by $a\in A$ between any pair of vertices of form $(g, g\cdot a)$. We denote such Cayley graph as $\Cay_A(G)$.
\end{df}

The next definition aims to define the main property of interest in this paper -- biautomaticity. This definitions can be found in the section 3 of the paper \cite{triangles_squares}.

\begin{df}[biautomatic group]
    Let $G$ be a group generated by set $S$ closed under inversion. The group is \textit{biautomatic} if there is a regular language $L$ over $S$ such that
    \begin{itemize}
        \item for every $g\in G$ there is a path in the Cayley graph $\Cay_S(G)$ starting in $1$ and ending in $g$ labeled by a word $w\in L$;
        \item set of paths labeled by words from $L$ in Cayley graph $\Cay_S(G)$ satisfies fellow traveler property.
    \end{itemize}
\end{df}

Now, we state a theorem that yield many important examples of biautomatic groups. Biautomaticity of systolic groups is one of the main motivations of the paper \cite{triangles_squares} as well as this paper. This theorem was originally proven in section 13 of \cite{systolic}, however it was later reproved using modern techniques in section 8.2 of \cite{path_systems}.

\begin{thm}[Januszkiewicz, Świątkowski]\label{systolic-biautomatic}
    \textnormal{Systolic groups}, that is, groups acting geometrically and cellularly on systolic complexes are biautomatic.
\end{thm}

We will now state the main theorem from the paper \cite{path_systems}. This theorem justifies analyzing of paths in $1$-skeleton of a complex on which $G$ acts geometrically and cellularly as a way of proving biautomaticity. More details about this theorem and the preliminary definition \ref{reg_path_sys} can be found in the aforementioned paper, particularly in sections 1 and 3.

\begin{df}[regular path system]\label{reg_path_sys}
    Let $G$ be a group acting geometrically on a graph $X$. An \textit{automaton over} $(G,X)$ is a directed graph $M$ (possibly infinite, with loops and multiple edges joining the same two points) together with two selected sets of vertices $I$ called the \textit{initial states} and $F$ called the \textit{final states}, a graph homomorphism $m: M \to X$ and a action of $G$ on $M$ by automorphisms respecting orientation of edges satisfying the following properties:
    \begin{itemize}
        \item[1)] the sets $I$ and $F$ are $G$-invariant;
        \item[2)] the map $m$ is $G$-equivariant, that is $g\cdot m(v) = m(g\cdot v)$ for any $v$ vertex of $M$ and any $g\in G$.
    \end{itemize}
    A path in $M$ is called \textit{internal} if it starts in one of the initial states and ends in the final states. The set $\mathcal{P}$ of images of internal paths in $M$ through the map $m$ is called the \textit{set of recognized paths}. An automaton over $(G,X)$ is called \textit{finite-to-one} if the preimage of each vertex $v$ from $X$ by the map $m$ is finite. A path system $\mathcal{P}$ in $X$ is called \textit{regular} if there exists a finite-to-one automaton over $(G,X)$ for which $\mathcal{P}$ is the set of recognized paths.
\end{df}

\begin{thm}[Świątkowski]\label{swiatkowskis_thm}
    Let $G$ be a group acting geometrically on a connected graph $X$. Moreover let $\mathcal{P}$ be a $G$-invariant path system in $X$ such that:
    \begin{itemize}
        \item $\mathcal{P}$ is regular;
        \item there is a vertex $v_0$ in $X$ such that each path in $\mathcal{P}$ starts and ends at vertices from the orbit $G\cdot v_0$ and for any two vertices from that orbit there is a path in $\mathcal{P}$ starting in one of those vertices and ending in the other one;
        \item $\mathcal{P}$ satisfies the generalized fellow traveler property.
    \end{itemize}
    Then $G$ is biautomatic.
\end{thm}

The following definition was stated by J. McCammond and R. Levitt as an attempt of proving biautomaticity of groups acting on CAT(0) triangle-square complexes. More context regarding Gersten-Short geodesics can be found in \cite{triangles_squares}, particularly in section 7.

\begin{df}[Gersten-Short geodesics]
    Let $X$ be a CAT(0) triangle-square complex and let $v, u$ be two vertices in $X$. It was proven in section 7 of \cite{triangles_squares} that all the geodesic paths between $v$ and $u$ in $1$-skeleton of $X$ satisfy one of the following conditions (formal description of those conditions is presented in definition 5.3 of \cite{triangles_squares} under the name moves):
    \begin{itemize}
        \item[(i)] all geodesic paths between $v$ and $u$ have the same first edge $e$. Let $v'$ be the vertex adjacent to $e$ other than $v$;
        \item[(ii)] there are two edges $e, e'$ contained in a single square cell $S$ such that the first edge of each of the geodesic paths is either $e$ or $e'$. In such case there are geodesic paths between $v$ and $u$ that travel through either of the edges $e$ and $e'$ and then through the vertex $v'$ in $S$ not belonging to either of the edges $e, e'$;
        \item[(iii)] there are two edges $e, e'$ contained in a single triangle cell $T$ such that first edge of the geodesic paths is either $e$ or $e'$. Then there is a triangle $T'$ and a finite (but possibly empty) row of squares with triangles $T$ and $T'$ on either end. Let $v'$ be a vertex of $T'$ that doesn't belong to the row of squares, then there are geodesic paths between $v$ and $u$ traveling on either side of the described above region built from square and triangles through the point $v'$.
    \end{itemize}
    We will now define the \textit{set of choke points} between $v$ and $u$ recursively:
    \begin{itemize}
        \item if $v=u$, then there is only one choke point $v$;
        \item if $u \neq v$, then we consider set of all geodesic paths between $v$ and $u$. From the note above they have to satisfy one of the above conditions (i), (ii) or (iii). Let us define $v'$ as above, then the set of choke points between $v$ and $u$ is equal to the union $\set{v}$ and of the set of choke points between $v'$ and $u$.
    \end{itemize}
    A \textit{Gersten-Short geodesic} between $v$ and $u$ is a geodesic path traveling through all the choke points between $v$ and $u$. A visualization of Gersten-Short geodesics are presented in figures 5 and 6 of \cite{triangles_squares}.
\end{df}

The following definitions were establish in order to describe possible behaviours of Gersten-Short geodesics inside flat fragments of triangle-square complexes. For more context see section 8 of \cite{triangles_squares}.

\begin{df}[flats]
    A \textit{flat} is a triangle-square complex isometric to the the euclidean plane $\RR^2$.  A vertex inside a flat $F$ is called a \textit{corner} if there are $3$ triangle cells $T_1, T_2, T_3$ and $2$ square cells $S_1, S_2$ such that $v$ is a vertex of $T_1, T_2, T_3, S_1, S_2$ and square cells $S_1, S_2$ do not share an edge. A \textit{region} in a flat $F$ is a maximal subcomplex $Y$ built out of one type of polygons (either only squares or only triangles) such that between any two shapes $P, P'$ in $Y$ there exists a sequence $P_0=P, P_1, \ldots, P_n= P' $ of shapes in $Y$ such that for each $0 \leqslant i < n$ shapes $P_i, P_{i+1}$ have a common edge.
\end{df}

\begin{df}[Types of flats]
    If a flat $F$ contains either only triangles or only squares then we say that $F$ is \textit{pure}. If such flat contain both squares and triangles but does not contain any corners then we say that $F$ is \textit{striped}. If such flat contain at least one corner vertex and at least one unbounded region then we say that $F$ is \textit{radial}. If such flat contain contain at least one corner vertex and all the regions in that flat are bounded then we say that $F$ is \textit{crumpled}. If $F$ is crumpled and there exists a bound on number of cells of each region of $F$ then we say that $F$ is \textit{thoroughly crumpled}. 
    
    We say that a flat is \textit{potentialy periodic} if it admits a cocompact, cellular action of $\ZZ^2$. If a flat does not admit such action then we say that it is \textit{intrinsically aperiodic}. Note that all radial flats are intrinsically aperiodic.
\end{df}

It is strongly believed, both by the author of this paper and by the authors of \cite{triangles_squares}, that Gersten-Short geodesics form a regular path system. The fellow traveling of such geodesics is much more problematic since, as was described in the theorem below, such geodesics may not fellow travel even inside a single flat. For details see chapters 9 and 10 of \cite{triangles_squares}.

\begin{thm}[Levitt, McCammond]\label{ftp-in-flats}
    There exists a constant $k>0$ such that for any pure or striped flat $F$ any two Gersten-Short geodesics starting and ending at most distance $1$ apart $k$-fellow-travel. Moreover, if $F$ is a potentially periodic flat then there exists a positive constant $k$ depending on $F$ such that any two Gersten-Short geodesics starting and ending at most distance $1$ apart $k$-fellow-travel. If $F$ is intrinsically aperiodic then there might not be any global constant $k$ satisfying the above property.
\end{thm}

Note that any potentially periodic flat is either pure or thoroughly crumpled. This observation allowed R. Levitt and J. McCammond to formulate two conjectures (9.4 and 11.1 in \cite{triangles_squares}) regarding behavior fundamental groups of CAT(0) triangle-square complexes, which were the main interest of this paper.

\begin{conj}[Intrinsically aperiodic flats] \label{conj:1}
    If $X$ is a compact locally CAT(0) triangle-square complex and $F$ is an intrinsically aperiodic flat then $F$ does not isometrically embed in $\wtl{X}$.
\end{conj}

\begin{conj}[Thoroughly crumpled planes] \label{conj:2}
    For any compact locally CAT(0) triangle-square complex $X$ there are only finitely many distinct (up to cellular isometries) thoroughly crumpled planes that embed in $\wtl{X}$.
\end{conj}

\section{Flats in triangle-square complexes} \label{flats_in_complexes}

\begin{lm} \label{flat_embedding}
    Let $F$ be a flat and let $X$ be a locally CAT(0) triangle-square complex. Flat $F$ embeds in the universal cover $\wtl{X}$ if there exists a cellular map $\varphi\colon F \to X$ locally isometric on each cell of $F$ such that for all vertices the induced map $\Lk(v,F) \to \Lk(\varphi(v), X)$ is injective.
\end{lm}

\begin{proof}
    Let $\wtl{\varphi}\colon F \to \wtl{X}$ be a lifting map, i.e. continuous map such that $\varphi = p\circ \wtl{\varphi}$. We need to show that the map $\wtl{\varphi}$ is an embedding, let us assume on the contrary that it is not. Since $\wtl{\varphi}$ is not an embedding, there are two different points $x, y\in F$ such that $\wtl{\varphi}(x) = \wtl{\varphi}(y)$, let $\gamma$ be a geodesic in euclidean space $F$ between $x$ and $y$. Since $\varphi$ is locally isometric on each cell and the induced map $\Lk(v,F) \to \Lk(\varphi(v), X)$ is injective then $\varphi$ is a local isometry. Moreover $p$ is also a local isometry, therefore $\wtl{\varphi}$ is a local isometry as well. Thus $\wtl{\varphi}\circ \gamma$ is a local geodesic in a CAT(0) space, therefore from fact \ref{simply_connected+local_geodesics} we know that it is a geodesic. But the endpoints of $\wtl{\varphi}\circ \gamma$ are the same point, therefore such geodesic has to have length $0$, which contradicts the fact that $\wtl{\varphi}$ is a local isometry and $\gamma$ has non-zero length.
\end{proof}

\begin{df}[The triangle-square flat $\mathbf{F}$]
    Below we present the formal constuction of flat $F$ presented in the figure \ref{flat_F}. Let $\mathcal{E}$ be the Eisenstein plane and let $r_0, r_1, \ldots, r_5$ be the $6$ pairwise distinct rays on $\mathcal{E}$ starting at $0$ and contained in the $1$-skeleton of $\mathcal{E}$. The flat $F$ is obtained by replacing each of the rays $r_0, r_1, \ldots, r_5$ by a strip $[0, \infty)\times[0,1]$ cubulated by unit squares and replacing the central vertex $0$ by a unit regular hexagon triangulated using unit equilateral triangles. For future use we will denote the middle point of the central hexagon of $F$ as $o$.
    
    \begin{figure}[!ht]
    \includegraphics[trim=10.5cm 20cm 10.5cm 20cm, clip, width=\textwidth]{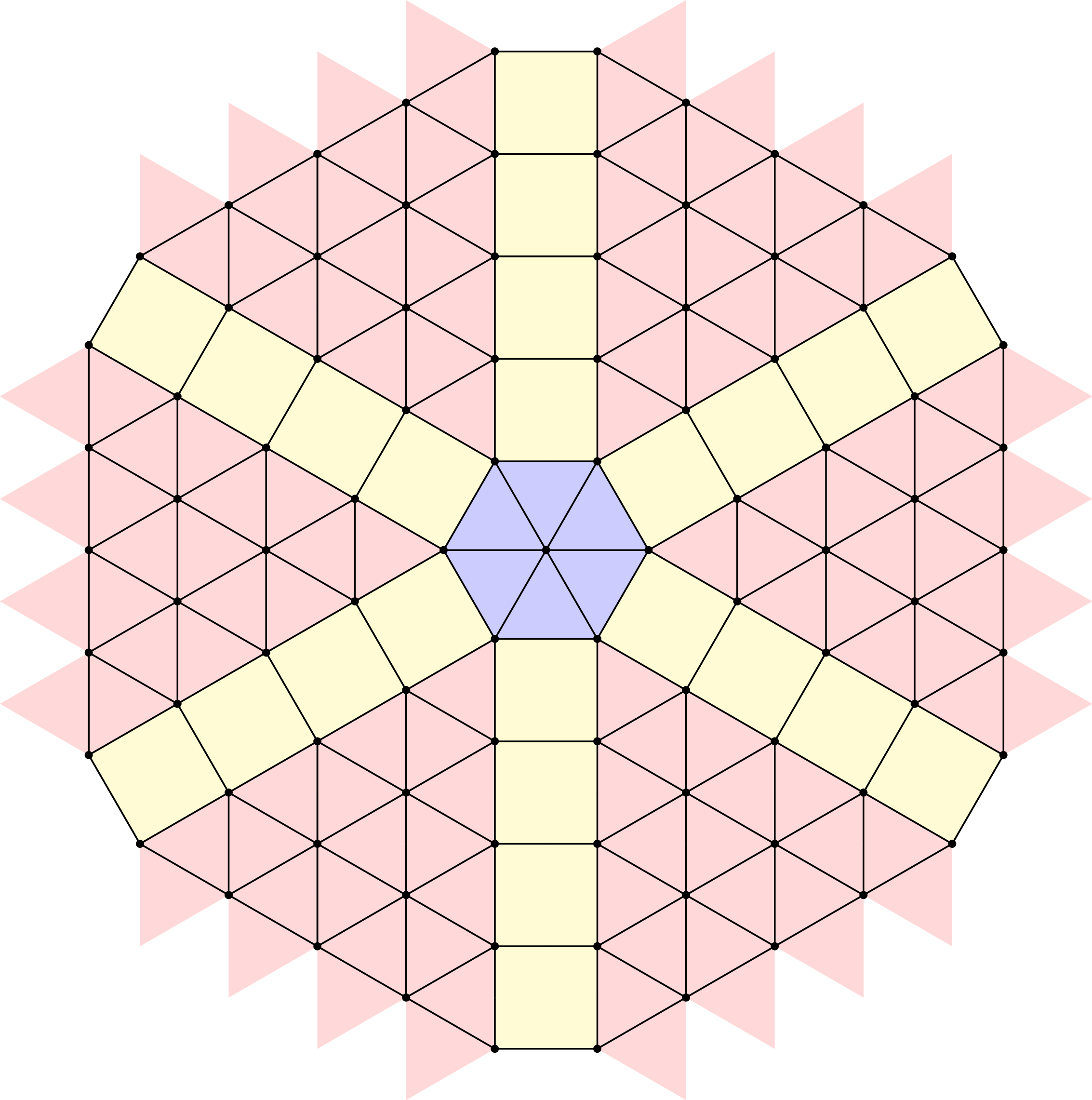}
    \caption{A part of a radial flat $F$.}
    \label{flat_F}
\end{figure}
\end{df}
    
    Since flat $F$ is radial, proving its existence in universal cover of any compact, locally CAT(0) triangle-square complex would automatically disprove conjecture \ref{conj:1}. Moreover, similarly to an example given in \cite{triangles_squares} (Proposition 9.3 in mentioned paper), Gersten-Short geodesics do not fellow travel.

\begin{prop}[Fellow traveling in flat $\mathbf{F}$] \label{fellow_traveling_in_F}
    The Gersten-Short geodesics do not fellow travel in the flat $F$.
\end{prop}
\begin{proof}
    Let $F$ be the flat described in the previous definition and let $u_1, u_2$ be two vertices, both at distance $\ell$ from $o$ in $1$-skeleton of $F$, connected by an edge in the interior of one of square regions. Moreover let $v_1, v_2$ be two vertices obtained by rotating vertices $u_2, u_1$ around the point $o$ by an angle of $\frac{2\pi}{3}$ in direction such that $d(u_1, v_1) < d(u_2, v_2)$. The figure \ref{Fellow_traveling_in_F_picture} illustrates described points and Gersten-Short geodesics between $u_1, v_1$ and $u_2, v_2$ for $\ell = 7$. Note that, similarly to the example described in proposition 9.3 of \cite{triangles_squares}, the Gersten-Short geodesic between $u_1$ and $v_1$ travels through the complex in a roughly straight line (which is roughly horizontal in the figure \ref{Fellow_traveling_in_F_picture}). In that configuration for any odd $\ell$ the distance in $1$-skeleton of $F$ between $o$ and any point on a Gersten-Short geodesic between $u_1$ and $v_1$ is at least $\frac{\ell+1}{2}$. Moreover the only geodesic, thus the Gersten-Short geodesic, between $u_2$ and $v_2$ in one skeleton of $F$ travels through the vertex $o$, therefore Gersten-Short geodesics cannot fellow travel on the entire flat $F$. 
\end{proof}

\begin{figure}[!ht]
    \includegraphics[trim=9cm 20cm 9cm 35cm, clip, width=\textwidth]{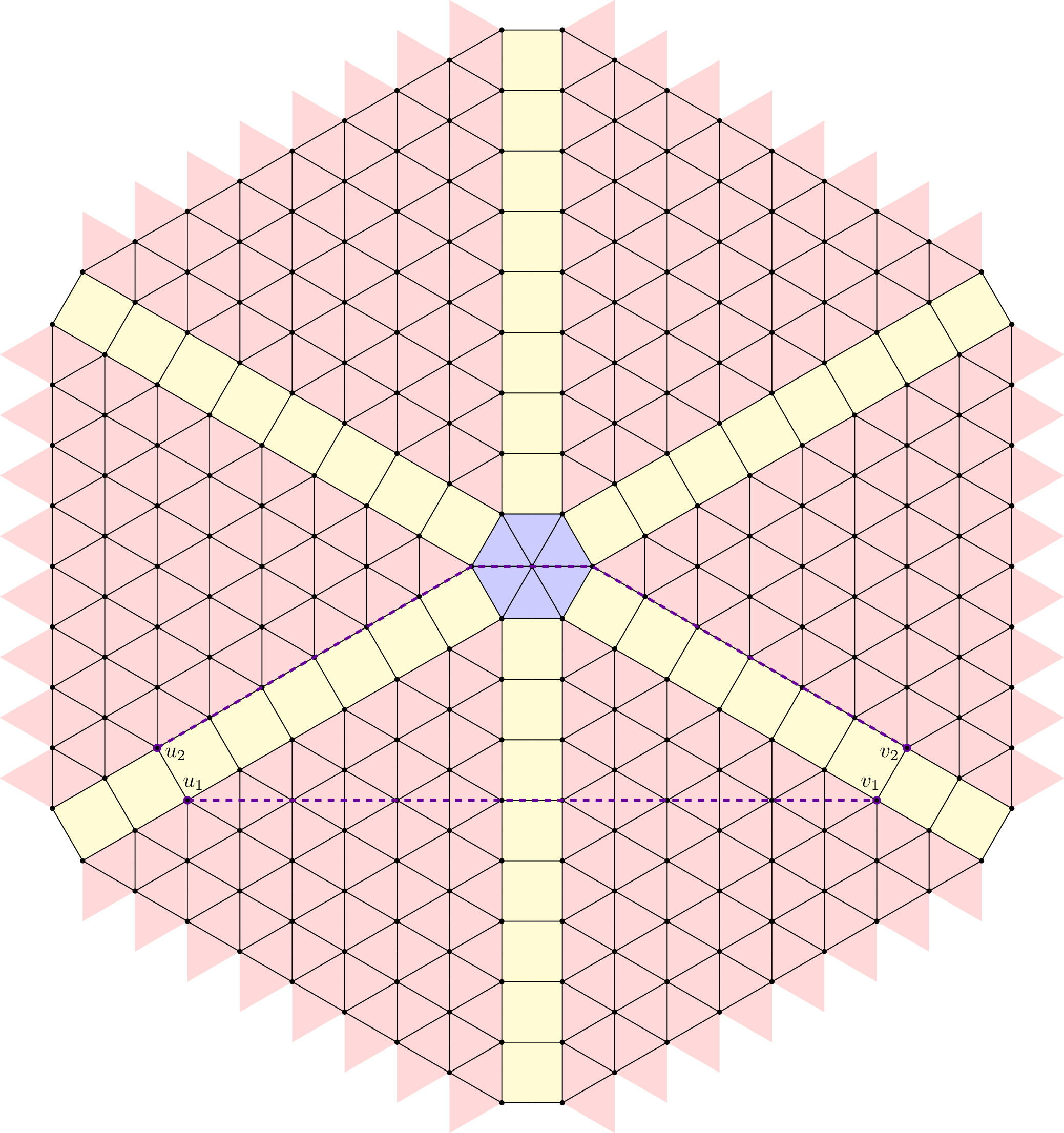}
    \caption{Points $u_1, u_2, v_1, v_2$ and Gersten-Short geodesics inside flat $F$.}
    \label{Fellow_traveling_in_F_picture}
\end{figure}

\begin{df}[The triangle-square flats $\mathbf{F_n}$]
    Let $\mathcal{E}$ be the Eisenstein plane. The flat $G$ is obtained by replacing each edge by an unit square and each vertex by an unit regular hexagon. The flat $F_n$ is obtained from flat $G$ described above by subdividing each square into $n$ rectangles with sides $1$ and $\frac{1}{n}$ each in such a way that each edge of length $1$ is parallel to an edge of a hexagon, subdividing each of the triangles into $n^2$ equilateral triangles with sides of length $\frac{1}{n}$ each, subdividing each of the hexagons into $6$ unit equilateral triangles length and adjusting the metric in such a way that each of the rectangles becomes a unit square and each of the triangles becomes a unit triangle (see Figure \ref{flat_F3}).
    
    \begin{figure}[!ht]
    \includegraphics[trim=10cm 75cm 10cm 75cm, clip,width=\textwidth]{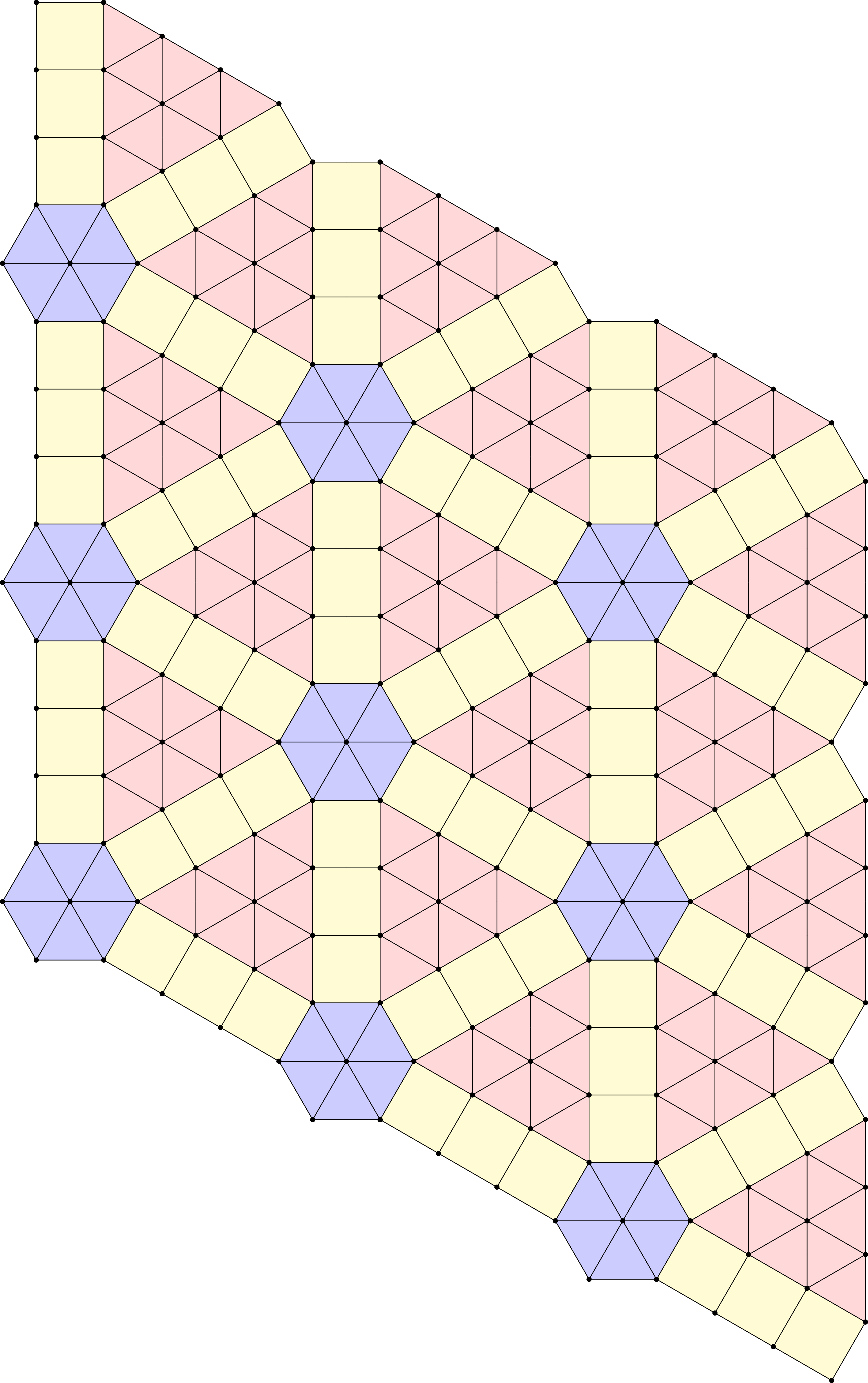}
    \caption{A part of a thoroughly crumpled flat $F_3$.}
    \label{flat_F3}
\end{figure}

    Since flats $F_n$ are parwise non-isomorphic thoroughly crumpled flats, for the purpose of disproving conjecture \ref{conj:2} it is enough to show that there exist a compact, locally CAT(0) triangle-square complex $X$ and an infinite subfamily of family $F_n$ such that each flat from that family embeds into the universal cover $\wtl{X}$. Note that any bounded part of flat $F$ is contained in all but finitely many of flats $F_n$, thus an argument from proposition \ref{fellow_traveling_in_F} can be replicated in some flat from family described above. Therefore existence of such family also automatically proves that Gersten-Short geodesics cannot fellow travel in~$\wtl{X}$.
\end{df}

\section{First Example}\label{first_example}

Let $X_1$ complex presented in the figure \ref{example_1} be a space consisting of three rhombi $K_1L_1M_1N_1$, $K_2L_2M_2N_2$, $ABCD$ with sides of unit length and such that $\sphericalangle N_1K_1L_1 = \sphericalangle N_2K_2L_2 = \sphericalangle ABC = \frac{\pi}{3}$ and three unit squares $P_1Q_1R_1S_1$, $P_2Q_2R_2S_2$ and $P_3Q_3R_3S_3$ glued together in the following way:
\begin{align*}
    e_1:M_1L_1\sim N_1K_1\sim P_1Q_1, && e_2:K_1L_1\sim N_1M_1\sim P_2Q_2, && e_3:L_1N_1\sim P_3Q_3,
\end{align*}
\begin{align*}
    f_1:L_2K_2\sim M_2N_2\sim S_1R_1, && f_2:L_2M_2\sim K_2N_2\sim S_2R_2, && f_3:N_2L_2\sim S_3R_3,
\end{align*}
\begin{align*}
    g:BA\sim CD\sim P_1S_1 \sim Q_1R_1, && h:CB \sim DA \sim P_2S_2 \sim Q_2R_2,
\end{align*}
\begin{align*}
    i:AC\sim P_3S_3 \sim Q_3R_3.
\end{align*}
\begin{figure}[!ht]
    \begin{center}
\begin{tikzpicture}
\fill[fill=red!15] (6,0.75)--(7.732,1.75)--(6,2.75);
\fill[fill=red!15] (6,0.75)--(4.268,1.75)--(6,2.75);
\draw[red, thick, ->] (6,0.75) -- (6,1.75);
\draw[red, thick] (6,1.75) -- (6,2.75);
\draw[black!50!green, thick, ->] (7.732,1.75) -- (6.866,1.25);
\draw[black!50!green, thick] (6,0.75) -- (6.866,1.25);
\draw[blue, thick, ->] (6,2.75) -- (6.866,2.25);
\draw[blue, thick] (7.732,1.75) -- (6.866,2.25);
\draw[blue, thick, ->] (4.268,1.75) -- (5.134,1.25);
\draw[blue, thick] (6,0.75) -- (5.134,1.25);
\draw[black!50!green, thick, ->] (6,2.75) -- (5.134,2.25);
\draw[black!50!green, thick] (5.134,2.25) -- (4.268,1.75);
\filldraw[black!50!green] (5.134,2.25) circle (0pt) node[anchor=south east]{$e_1$};
\filldraw[black!50!green] (6.866,1.25) circle (0pt) node[anchor=north west]{$e_1$};
\filldraw[blue] (6.866,2.25) circle (0pt) node[anchor=south west]{$e_2$};
\filldraw[blue] (5.134,1.25) circle (0pt) node[anchor=north east]{$e_2$};
\filldraw[red] (6,1.75) circle (0pt) node[anchor=east]{$e_3$};
\filldraw[black] (7.732,1.75) circle (2pt) node[anchor=west]{$M_1$};
\filldraw[black] (4.268,1.75) circle (2pt) node[anchor=east]{$K_1$};
\filldraw[black] (6,0.75) circle (2pt) node[anchor=north]{$L_1$};
\filldraw[black] (6,2.75) circle (2pt) node[anchor=south]{$N_1$};

\fill[fill=red!15] (6,-0.75)--(7.732,-1.75)--(6,-2.75);
\fill[fill=red!15] (6,-0.75)--(4.268,-1.75)--(6,-2.75);
\draw[red, thick, ->>] (6,-0.75) -- (6,-1.75);
\draw[red, thick] (6,-1.75) -- (6,-2.75);
\draw[black!50!green, thick, ->>] (7.732,-1.75) -- (6.866,-1.25);
\draw[black!50!green, thick] (6,-0.75) -- (6.866,-1.25);
\draw[blue, thick, ->>] (6,-2.75) -- (6.866,-2.25);
\draw[blue, thick] (7.732,-1.75) -- (6.866,-2.25);
\draw[blue, thick, ->>] (4.268,-1.75) -- (5.134,-1.25);
\draw[blue, thick] (6,-0.75) -- (5.134,-1.25);
\draw[black!50!green, thick, ->>] (6,-2.75) -- (5.134,-2.25);
\draw[black!50!green, thick] (5.134,-2.25) -- (4.268,-1.75);
\filldraw[black!50!green] (5.134,-2.25) circle (0pt) node[anchor=north east]{$f_1$};
\filldraw[black!50!green] (6.866,-1.25) circle (0pt) node[anchor=south west]{$f_1$};
\filldraw[blue] (6.866,-2.25) circle (0pt) node[anchor=north west]{$f_2$};
\filldraw[blue] (5.134,-1.25) circle (0pt) node[anchor=south east]{$f_2$};
\filldraw[red] (6,-1.75) circle (0pt) node[anchor=east]{$f_3$};
\filldraw[black] (7.732,-1.75) circle (2pt) node[anchor=west]{$M_2$};
\filldraw[black] (4.268,-1.75) circle (2pt) node[anchor=east]{$K_2$};
\filldraw[black] (6,-0.75) circle (2pt) node[anchor=south]{$N_2$};
\filldraw[black] (6,-2.75) circle (2pt) node[anchor=north]{$L_2$};

\fill[fill=blue!20] (8,0)--(9,1.732)--(10,0)--(9,-1.732);

\draw[black, thick, ->>>] (8,0)--(9,0);
\draw[black, thick] (9,0)--(10,0);
\filldraw[black] (9,0) circle (0pt) node[anchor=south]{$i$};

\draw[black, thick, ->>] (9,1.732)--(8.5,0.866);
\draw[black, thick] (8,0)--(8.5,0.866);
\draw[black, thick] (9,-1.732)--(9.5,-0.866);
\draw[black, thick, ->>] (10,0)--(9.5,-0.866);
\filldraw[black] (8.5,0.866) circle (0pt) node[anchor=south east]{$h$};
\filldraw[black] (9.5,-0.866) circle (0pt) node[anchor=north west]{$h$};

\draw[black, thick] (8,0)--(8.5,-0.866);
\draw[black, thick, ->] (9,-1.732)--(8.5,-0.866);
\draw[black, thick] (9.5,0.866)--(9,1.732);
\draw[black, thick, ->] (10,0)--(9.5,0.866);
\filldraw[black] (8.5,-0.866) circle (0pt) node[anchor=north east]{$g$};
\filldraw[black] (9.5,0.866) circle (0pt) node[anchor=south west]{$g$};

\filldraw[black] (10,0) circle (2pt) node[anchor=west]{$C$};
\filldraw[black] (9,1.732) circle (2pt) node[anchor=south]{$D$};
\filldraw[black] (8, 0) circle (2pt) node[anchor=east]{$A$};
\filldraw[black] (9,-1.732) circle (2pt) node[anchor=north]{$B$};

\fill[fill=yellow!20] (11,0.75)--(13,0.75)--(13,2.75)--(11,2.75);
\draw[black, thick, ->] (11,0.75)--(11,1.75);
\draw[black, thick] (11,1.75)--(11,2.75);
\filldraw[black] (11,1.75) circle (0pt) node[anchor=east]{$g$};
\draw[black, thick, ->] (13,0.75)--(13,1.75);
\draw[black, thick] (13,1.75)--(13,2.75);
\filldraw[black] (13,1.75) circle (0pt) node[anchor=west]{$g$};
\draw[black!50!green, thick, ->] (11,0.75)--(12,0.75);
\draw[black!50!green, thick] (12,0.75)--(13,0.75);
\filldraw[black!50!green] (12,0.75) circle (0pt) node[anchor=north]{$e_1$};
\draw[black!50!green, thick, ->>] (11,2.75)--(12,2.75);
\draw[black!50!green, thick] (12,2.75)--(13,2.75);
\filldraw[black!50!green] (12,2.75) circle (0pt) node[anchor=south]{$f_1$};
\filldraw[black] (11,0.75) circle (2pt) node[anchor=north east]{$P_1$};
\filldraw[black] (13,0.75) circle (2pt) node[anchor=north west]{$Q_1$};
\filldraw[black] (13,2.75) circle (2pt) node[anchor=south west]{$R_1$};
\filldraw[black] (11,2.75) circle (2pt) node[anchor=south east]{$S_1$};

\fill[fill=yellow!20] (11,-0.75)--(13,-0.75)--(13,-2.75)--(11,-2.75);
\draw[black, thick, ->>] (11,-2.75)--(11,-1.75);
\draw[black, thick] (11,-1.75)--(11,-0.75);
\filldraw[black] (11,-1.75) circle (0pt) node[anchor=east]{$h$};
\draw[black, thick, ->>] (13,-2.75)--(13,-1.75);
\draw[black, thick] (13,-1.75)--(13,-0.75);
\filldraw[black] (13,-1.75) circle (0pt) node[anchor=west]{$h$};
\draw[blue, thick, ->] (11,-2.75)--(12,-2.75);
\draw[blue, thick] (12,-2.75)--(13,-2.75);
\filldraw[blue] (12,-2.75) circle (0pt) node[anchor=north]{$e_2$};
\draw[blue, thick, ->>] (11,-0.75)--(12,-0.75);
\draw[blue, thick] (12,-0.75)--(13,-0.75);
\filldraw[blue] (12,-0.75) circle (0pt) node[anchor=south]{$f_2$};
\filldraw[black] (11,-2.75) circle (2pt) node[anchor=north east]{$P_2$};
\filldraw[black] (13,-2.75) circle (2pt) node[anchor=north west]{$Q_2$};
\filldraw[black] (13,-0.75) circle (2pt) node[anchor=south west]{$R_2$};
\filldraw[black] (11,-0.75) circle (2pt) node[anchor=south east]{$S_2$};

\fill[fill=yellow!20] (14.5,-1)--(16.5,-1)--(16.5,1)--(14.5,1);
\draw[black, thick, ->>>] (14.5,-1)--(14.5,0);
\draw[black, thick] (14.5,0)--(14.5,1);
\filldraw[black] (14.5,0) circle (0pt) node[anchor=east]{$i$};
\draw[black, thick, ->>>] (16.5,-1)--(16.5,0);
\draw[black, thick] (16.5,0)--(16.5,1);
\filldraw[black] (16.5,0) circle (0pt) node[anchor=west]{$i$};
\draw[red, thick, ->] (14.5,-1)--(15.5,-1);
\draw[red, thick] (15.5,-1)--(16.5,-1);
\filldraw[red] (15.5,-1) circle (0pt) node[anchor=north]{$e_3$};
\draw[red, thick, ->>] (14.5,1)--(15.5,1);
\draw[red, thick] (15.5,1)--(16.5,1);
\filldraw[red] (15.5,1) circle (0pt) node[anchor=south]{$f_3$};
\filldraw[black] (14.5,-1) circle (2pt) node[anchor=north east]{$P_3$};
\filldraw[black] (16.5,-1) circle (2pt) node[anchor=north west]{$Q_3$};
\filldraw[black] (16.5,1) circle (2pt) node[anchor=south west]{$R_3$};
\filldraw[black] (14.5,1) circle (2pt) node[anchor=south east]{$S_3$};

\end{tikzpicture}
\end{center}
    \caption{The polyhedral complex $X_1$}
    \label{example_1}
\end{figure}
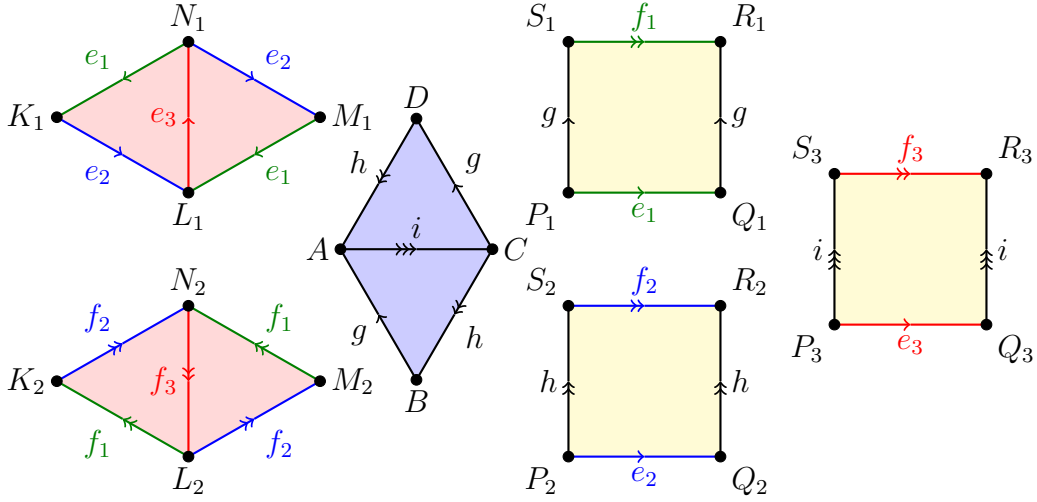

Note that by subdividing each of the rhombi from the definition above into two unit triangles we can in fact treat $X_1$ as an triangle-square complex with only one vertex $v$. Now we can prove the following lemma regarding geometry of $X_1$:

\begin{lm}
    The universal cover $\wtl{X_1}$ is CAT(0).
\end{lm}

\begin{proof}
    By using theorems \ref{GLC} and \ref{CH} we know that it is enough to check that each injective loop in link $\Lk(v,X_1)$ has length at least $2\pi$. This link is presented in Figure \ref{link_1}. It is clear that there are no isometric cycles of length less then $2\pi$, thus $X_1$ is locally CAT(0) and $\wtl{X_1}$ is CAT(0).
\end{proof}

\begin{figure}[!ht]
    \begin{center}
    \begin{tikzpicture}[scale=1.1]
    \draw[black, thick] (-1,0)--(-0.5,-0.866)--(0.5,-0.866)--(1,0)--(0.5,0.866)--(-0.5,0.866)--(-1,0);
    \draw[black, thick] (-1.5,-0.866)--(0,-1.732)--(1.5,-0.866)--(1.5,0.866)--(0,1.732)--(-1.5,0.866)--(-1.5,-0.866);
    \draw[black, thick] (-3.232, 0)--(-1.5,-2.866)--(1.5,-2.866)--(3.232, 0)--(1.5,2.866)--(-1.5,2.866)--(-3.232, 0);

    \path[red, thick, dashdotted] (1.5,-0.866) edge[bend right=30] (0.5,0.866);
    \path[red, thick, dashdotted] (1.5,-0.866) edge[bend left=30] (-0.5,-0.866);
    \path[red, thick, dashdotted] (-1.5,-0.866) edge[bend left=30] (-0.5,0.866);
    \path[red, thick, dashdotted] (-1.5,-0.866) edge[bend right=30] (0.5,-0.866);
    \path[red, thick, dashdotted] (0,1.732) edge[bend right=30] (-1,0);
    \path[red, thick, dashdotted] (0,1.732) edge[bend left=30] (1,0);
    
    \path[red, thick, dashdotted] (-3.232,0) edge[bend right=30] (0,-1.732);
    \path[red, thick, dashdotted] (3.232,0) edge[bend left=30] (0,-1.732);
    \path[red, thick, dashdotted] (-1.5,-2.866) edge[bend left=30] (-1.5,0.866);
    \path[red, thick, dashdotted] (1.5,2.866) edge[bend right=30] (-1.5,0.866);
    \path[red, thick, dashdotted, dashdotted] (1.5,-2.866) edge[bend right=30] (1.5,0.866);
    \path[red, thick, dashdotted] (-1.5,2.866) edge[bend left=30] (1.5,0.866);
    \filldraw[black] (-1,0) circle (1.5pt) node[anchor=west]{$e_2$};
    \filldraw[black] (-0.5,-0.866) circle (1.5pt) node[anchor=south west]{$\ol{e_1}$};
    \filldraw[black] (0.5,-0.866) circle (1.5pt) node[anchor=south east]{$e_3$};
    \filldraw[black] (1,0) circle (1.5pt) node[anchor=east]{$\ol{e_2}$};
    \filldraw[black] (0.5,0.866) circle (1.5pt) node[anchor=north east]{$e_1$};
    \filldraw[black] (-0.5,0.866) circle (1.5pt) node[anchor=north west]{$\ol{e_3}$};
    \filldraw[black] (0,-1.732) circle (1.5pt) node[anchor=south]{$\ol{h}$};
    \filldraw[black] (-1.5,0.866) circle (1.5pt) node[anchor=north west]{$\ol{g}$};
    \filldraw[black] (1.5,0.866) circle (1.5pt) node[anchor=north east]{$\ol{i}$};
    \filldraw[black] (0,1.732) circle (1.5pt) node[anchor=south]{$h$};
    \filldraw[black] (1.5,-0.866) circle (1.5pt) node[anchor=north west]{$g$};
    \filldraw[black] (-1.5,-0.866) circle (1.5pt) node[anchor=north east]{$i$};
    \filldraw[black] (-3.232,0) circle (1.5pt) node[anchor=west]{$\ol{f_2}$};
    \filldraw[black] (-1.5,-2.866) circle (1.5pt) node[anchor=south west]{$f_1$};
    \filldraw[black] (1.5,-2.866) circle (1.5pt) node[anchor=south east]{$\ol{f_3}$};
    \filldraw[black] (3.232,0) circle (1.5pt) node[anchor=east]{$f_2$};
    \filldraw[black] (1.5,2.866) circle (1.5pt) node[anchor=north east]{$\ol{f_1}$};
    \filldraw[black] (-1.5,2.866) circle (1.5pt) node[anchor=north west]{$f_3$};
    \end{tikzpicture}\hspace{0.5cm}
    \begin{tikzpicture}[scale=1.1]
    \filldraw[black] (-1,0) circle (1.5pt) node[anchor=west]{$e_2'$};
    \filldraw[black] (-0.5,-0.866) circle (1.5pt) node[anchor=south west]{$\ol{e_1'}$};
    \filldraw[black] (0.5,-0.866) circle (1.5pt) node[anchor=south east]{$e_3'$};
    \filldraw[black] (1,0) circle (1.5pt) node[anchor=east]{$\ol{e_2'}$};
    \filldraw[black] (0.5,0.866) circle (1.5pt) node[anchor=north east]{$e_1'$};
    \filldraw[black] (-0.5,0.866) circle (1.5pt) node[anchor=north west]{$\ol{e_3'}$};
    \filldraw[black] (0,-1.732) circle (1.5pt) node[anchor=south]{$\ol{h'}$};
    \filldraw[black] (-1.5,0.866) circle (1.5pt) node[anchor=north west]{$\ol{g'}$};
    \filldraw[black] (1.5,0.866) circle (1.5pt) node[anchor=north east]{$\ol{i'}$};
    \filldraw[black] (0,1.732) circle (1.5pt) node[anchor=south]{$h'$};
    \filldraw[black] (1.5,-0.866) circle (1.5pt) node[anchor=west]{$g'$};
    \filldraw[black] (-1.5,-0.866) circle (1.5pt) node[anchor=north east]{$i'$};
    \filldraw[black] (-3.232,0) circle (1.5pt) node[anchor=west]{$\ol{f_2'}$};
    \filldraw[black] (-1.5,-2.866) circle (1.5pt) node[anchor=south west]{$f_1'$};
    \filldraw[black] (1.5,-2.866) circle (1.5pt) node[anchor=south east]{$\ol{f_3'}$};
    \filldraw[black] (3.232,0) circle (1.5pt) node[anchor=east]{$f_2'$};
    \filldraw[black] (1.5,2.866) circle (1.5pt) node[anchor=north east]{$\ol{f_1'}$};
    \filldraw[black] (-1.5,2.866) circle (1.5pt) node[anchor=north west]{$f_3'$};
    \draw[black, thick] (-1,0)--(-0.5,-0.866)--(0.5,-0.866)--(1,0)--(0.5,0.866)--(-0.5,0.866)--(-1,0);
    \draw[black, thick] (-1.5,-0.866)--(0,-1.732)--(1.5,-0.866)--(1.5,0.866)--(0,1.732)--(-1.5,0.866)--(-1.5,-0.866);
    \draw[black, thick] (-3.232, 0)--(-1.5,-2.866)--(1.5,-2.866)--(3.232, 0)--(1.5,2.866)--(-1.5,2.866)--(-3.232, 0);

    \path[black, thick] (1.5,-0.866) edge[bend right=30] (0.5,0.866);
    \draw[black, thick] (-1.5,-0.866)--(-1.218,0.125)--(-0.5,0.866);
    \filldraw[black] (-1.218,0.125) circle (1.5pt);
    \filldraw[white] (-1.218,0.125) circle (1pt);
    \draw[black, thick] (-0.5,-0.866)--(0.5,-1.166)--(1.5,-0.866);
    \filldraw[black] (0.5,-1.166) circle (1.5pt);
    \filldraw[white] (0.5,-1.166) circle (1pt);
    \path[black, thick] (-1.5,-0.866) edge[bend right=30] (0.5,-0.866);
    \path[black, thick] (0,1.732) edge[bend right=30] (-1,0);
    \draw[black, thick] (0,1.732)--(0.718,0.991)--(1,0);
    \filldraw[black] (0.718,0.991) circle (1.5pt);
    \filldraw[white] (0.718,0.991) circle (1pt);
    
    \path[black, thick] (-3.232,0) edge[bend right=30] (0,-1.732);
    \draw[black, thick] (3.232,0)--(1.7,-1.4)--(0,-1.732);
    \filldraw[black] (1.7,-1.4) circle (1.5pt);
    \filldraw[white] (1.7,-1.4) circle (1pt);
    \draw[black, thick] (-1.5,-2.866)--(-2,-0.866)--(-1.5,0.866);
    \filldraw[black] (-2,-0.866) circle (1.5pt);
    \filldraw[white] (-2,-0.866) circle (1pt);
    \path[black, thick] (1.5,2.866) edge[bend right=30] (-1.5,0.866);
    \path[black, thick] (1.5,-2.866) edge[bend right=30] (1.5,0.866);
    \draw[black, thick] (-1.5,2.866)--(0.25,2.198)--(1.5,0.866);
    \filldraw[black] (0.25,2.198) circle (1.5pt);
    \filldraw[white] (0.25,2.198) circle (1pt);
    \end{tikzpicture}
    \end{center}
    \caption{Left: link of $v$ in $X_1$, right: link of $v'$ in $X_1'$.}
    \label{link_1}
\end{figure}
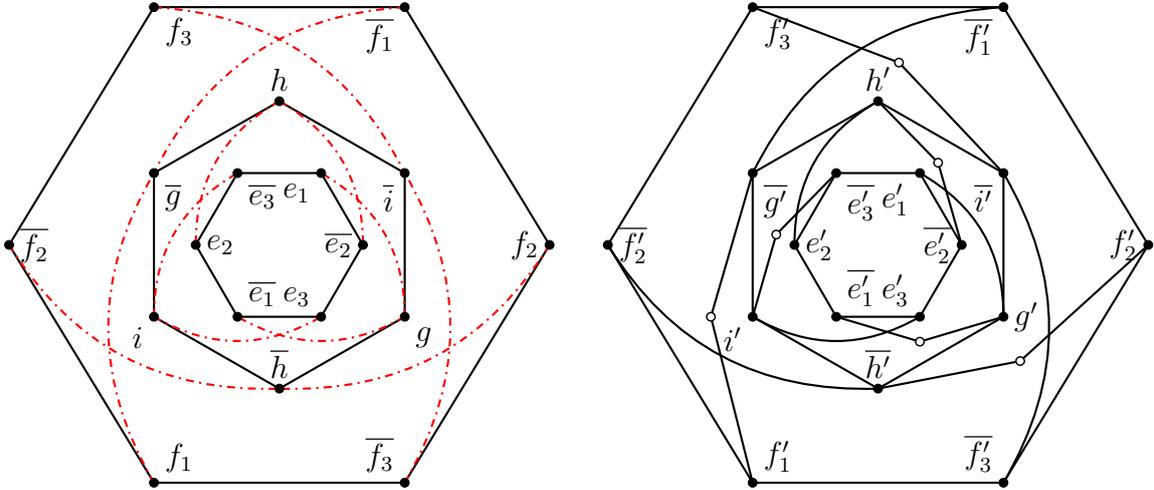

Moreover, it turns out that $X_1$ also violates both of the conjectures of Levitt and McCammond. 

\begin{lm}\label{F_embeds_into_X1}
    The radial flat $F$ embeds in the universal cover $\wtl{X_1}$.
\end{lm}

\begin{proof}
    By lemma \ref{flat_embedding} to show that $F$ embeds in $\wtl{X_1}$ it is enough to check that there is a cellular map of triangle-square complexes $\varphi\colon F\to X_1$ locally isometric on each cell such that for each vertex $w$ of $F$ the induced map $\Lk(w,F) \to \Lk(\varphi(w),X_1)$ is injective. Since each cell of $X_1$ is uniquely determined by oriented edges on its boundary, we can define such map by defining it on each edge of $F$. Such map is presented in figure \ref{flat_1_1_paper} and its formal description is presented below.
    
    Let us set the Cartesian coordinates on flat $F$ in such a way that $o$ is the center of the coordinate system and the central hexagon of $F$ has an edge parallel to the line $y=0$. Now let each edge of form $((x,y), (x-1,y))$ be mapped to an edge $g$, each edge of form $\left((x,y),(x+1/2,y+\sqrt{3}/2\right)$ be mapped to an edge $h$ and each edge of form $((x,y),(x+1/2,y-\sqrt{3}/2))$ be mapped to an edge $i$. Moreover let $E_1$ be the union of unbounded triangle regions containing points $(-1,0), (1/2, -\sqrt{3}/2)$ and $(1/2, \sqrt{3}/2)$; let each edge in $E_1$ of form $((x,y),(x,y-1))$ be mapped to $e_1$, let each edge in $E_1$ of form $((x,y),(x-\sqrt{3}/2,y+1/2))$ be mapped to $e_2$ and let each edge in $E_1$ of form $((x,y),(x+\sqrt{3}/2,y+1/2))$ be mapped to $e_3$. Finally let $E_2$ be the union of unbounded triangle regions containing points $(1,0), (-1/2, -\sqrt{3}/2)$ and $(-1/2, \sqrt{3}/2)$; let each edge in $E_2$ of form $((x,y),(x,y-1))$ be mapped to $f_1$, let each edge in $E_2$ of form $((x,y),(x-\sqrt{3}/2,y+1/2))$ be mapped to $f_2$ and let each edge in $E_2$ of form $((x,y),(x+\sqrt{3}/2,y+1/2))$ be mapped to $f_3$.
    
    \begin{figure}[!ht]
    \centering
    \includegraphics[trim=1.1cm 3.7cm 1.1cm 4.2cm, clip, width=\textwidth]{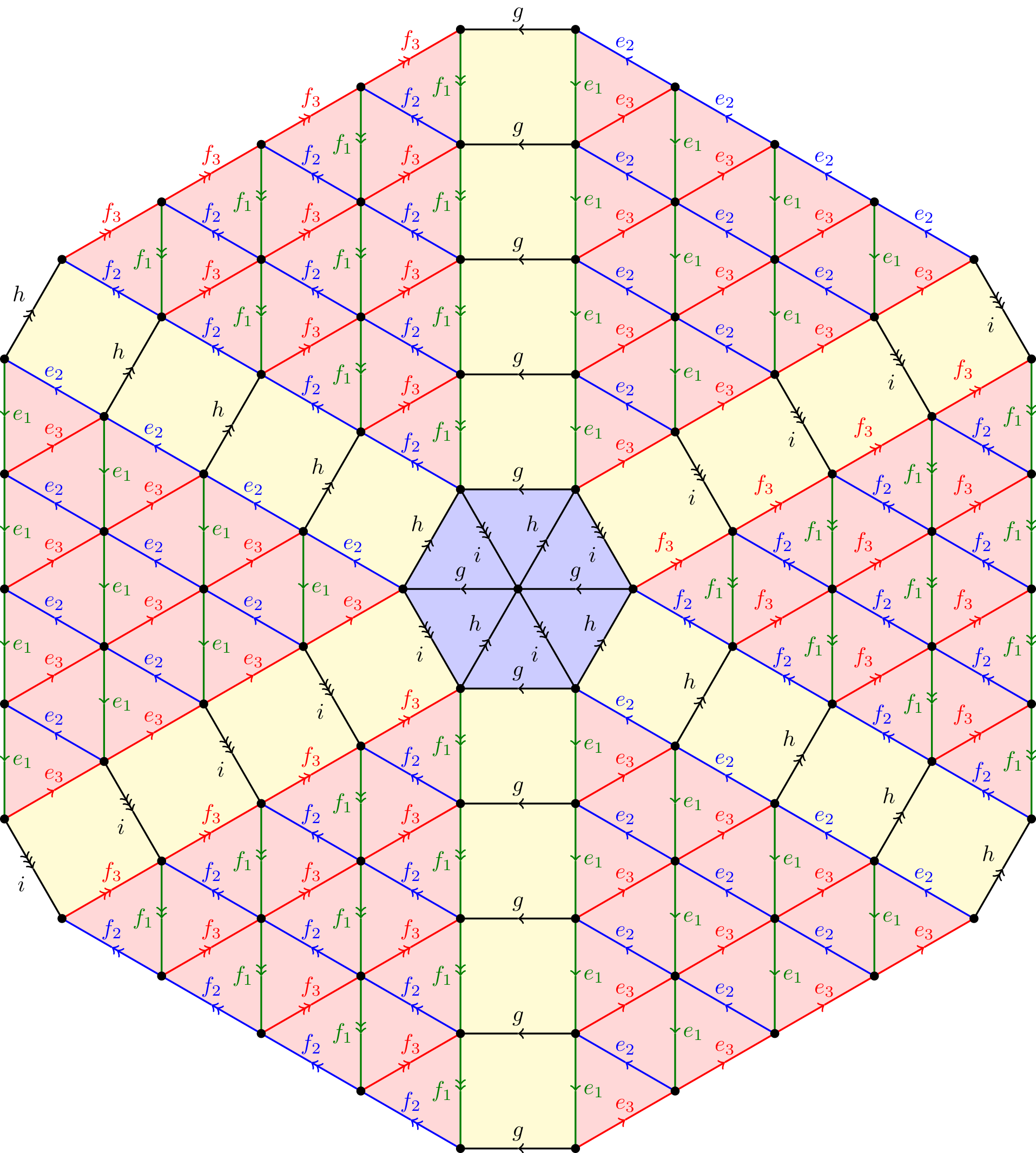}
    \caption{A part of a radial flat $F$ that embeds into $\wtl{X_1}$.}
    \label{flat_1_1_paper}
\end{figure}
    
    It is easy to check that for every vertex $w$ of $F$ the induced map of the links is injective. Therefore $F$ embeds in the universal cover $\wtl{X_1}$.

\end{proof}

\begin{lm}\label{Fn_embeds_into_X1}
    Each of the flats $F_n$ for $n\in \NN_+$ embeds in the universal cover $\wtl{X_1}$.
\end{lm}

\begin{proof}
    By using the same argument as in the proof of the previous lemma it is enough to check that there is a cellular map $\varphi_n\colon F_n\to X_1$ of triangle-square complexes such that for each vertex $w$ of $F_n$ the induced map of links $\Lk(w,F_n) \to \Lk(\varphi_n(w),X_1)$ injective. Similarly to the proof of the previous lemma it is enough to define such map on all edges of $F_n$.

    Let $D_n$ be a shape consisting of an unit hexagon, three rectangles, each of which has sides of length $1$ and $n$ and two equilateral triangles with sides of length $n$ glued together in a following way: each of some three consecutive edges of the hexagon are glued respectively to three edges of length $1$ of three rectangles and two pairs of two consecutive edges of length $n$ of polygon obtained in the previous step are glued respectively to two consecutive edges of two triangles. The obtained shape is an euclidean polygon (shown for $n = 2$ in Figure \ref{polygon_D2}).

    \begin{figure}[!ht]
    \includegraphics[width = 0.29\textwidth]{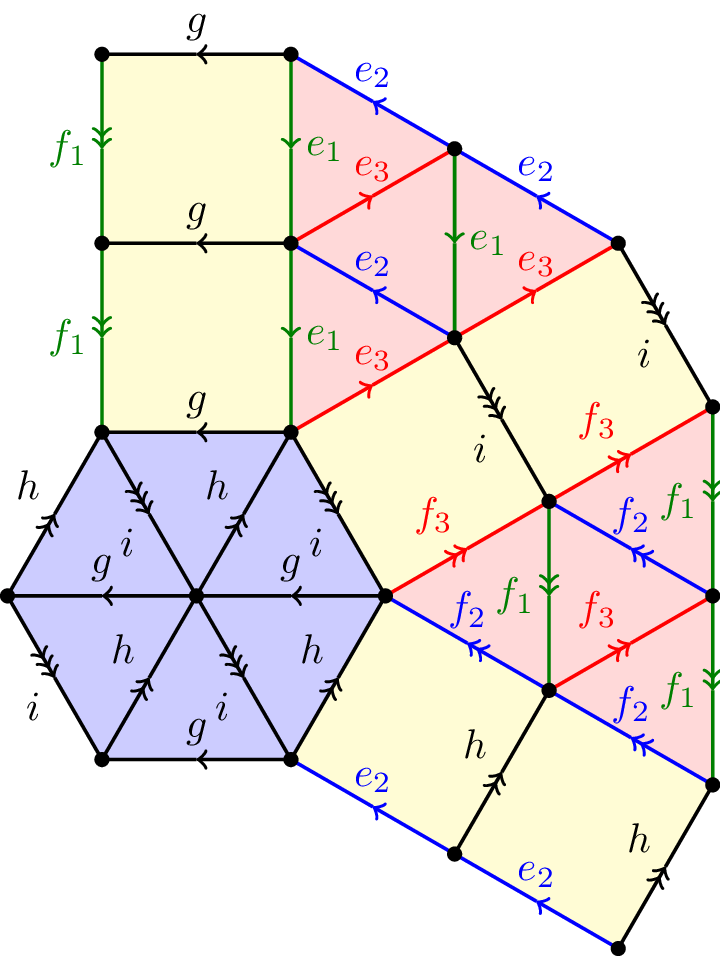}\hspace{0.5cm}
    \includegraphics[trim=1cm 8cm 1cm 8cm, clip, width=0.66\textwidth]{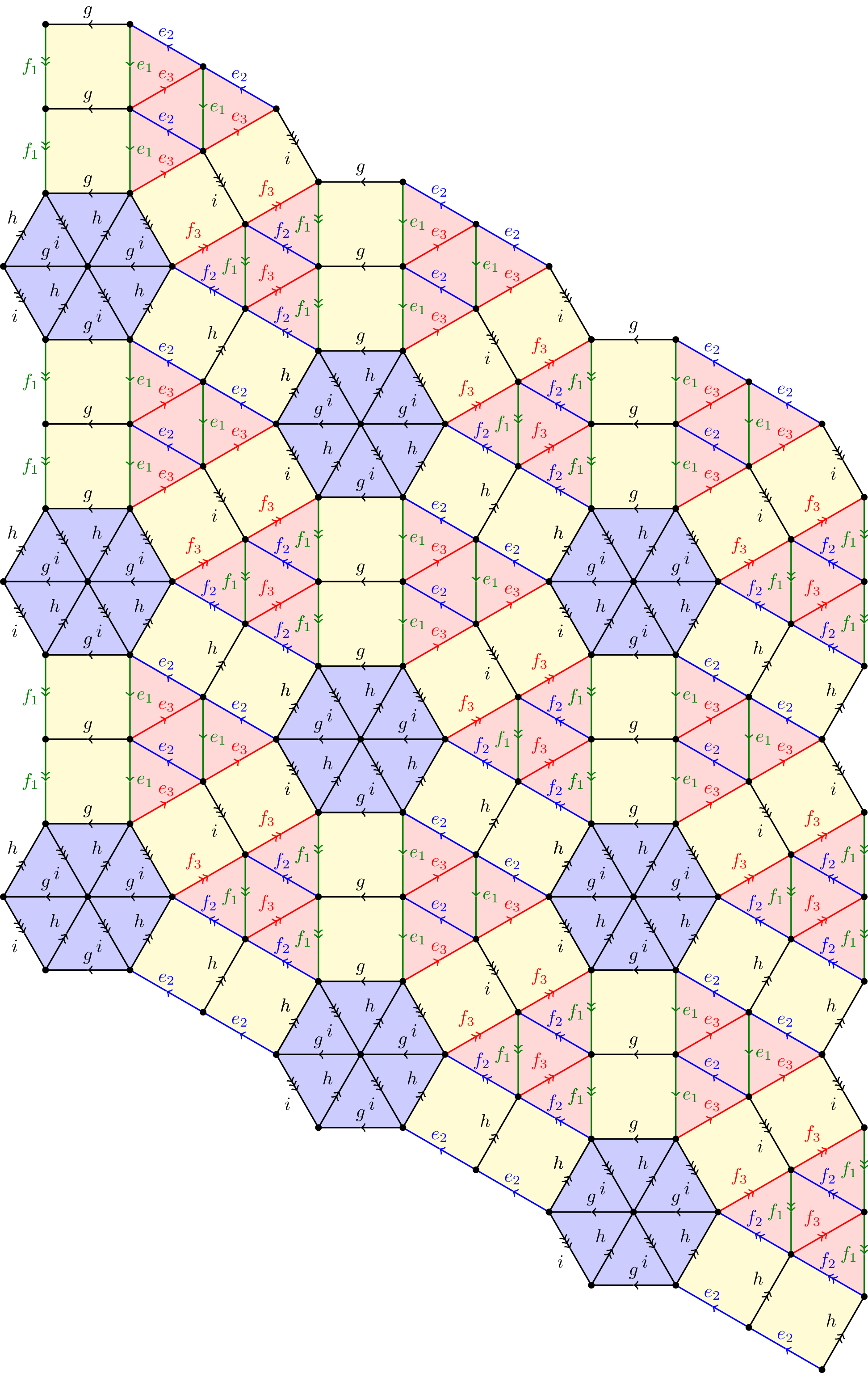}
    \caption{Left: polygon $D_2$ with a mapping $\psi_2$, right: a flat $F_2$ with a map $\varphi_2$.}
    \label{polygon_D2}
\end{figure}
    
    Note that the flat $F_n$ can be tiled by a translations of $D_n$, we will indicate the set of such translations as $\mathcal{T}$. Moreover the polygon $D_n$ as a triangle-square complex can be also obtained as a subcomplex of the flat $F$. Let us set $D_n$ to be a concrete subcomplex, now we can define a map $\psi_n$ on $D_n$ as restriction of the map $\varphi$ from the proof of the lemma \ref{F_embeds_into_X1} and the map $\varphi_n$ as 
    $$\varphi_n(x) = \psi_n(t^{-1}(x))\text{, where }t\in \mathcal{T}\text{ is such that }x\in t(D_n).$$

    The map $\varphi_n$ is presented for $n=2$ in Figure \ref{polygon_D2}. It is easy to check that such map is in fact well defined and that the induced map of the links is injective. Therefore $F_n$ embeds into the universal cover $\wtl{X_1}$.
\end{proof}

Even though the triangle-square complex $X_1$ violates both conjectures of Levitt and McCammond it turns out that $\pi_1(X_1)$ is biautomatic. We will show that we can modify the complex $X_1$ in a homeomorphic way and obtain a triangle complex $X_1'$ with systolic universal cover. 

\begin{lm}
    There is a triangle complex $X_1'$ homeomorphic to $X_1$ and such that $\wtl{X_1'}$ is systolic.
\end{lm}

\begin{proof}
    Again, we will directly construct the complex $X_1'$, presented in the figure \ref{example_1_systolization}. Let $X_1'$ be a space consisting of six rhombi $K_1'L_1'M_1'N_1'$, $K_2'L_2'M_2'N_2'$, $A'B'C'D'$, $P_1'Q_1'R_1'S_1'$, $P_2'Q_2'R_2'S_2'$ and $P_3'Q_3'R_3'S_3'$ with sides of unit length and such that $\sphericalangle N_1'K_1'L_1' = \sphericalangle N_2'K_2'L_2' = \sphericalangle A'B'C' = \sphericalangle S_1'P_1'Q_1' = \sphericalangle S_2'P_2'Q_2' = \sphericalangle S_3'P_3'Q_3' = \frac{\pi}{3}$ glued together in such a way:
    \begin{align*}
        e_1':M_1'L_1'\sim N_1'K_1'\sim P_1'Q_1', && e_2':K_1'L_1'\sim N_1'M_1'\sim P_2'Q_2', && e_3':L_1'N_1'\sim P_3'Q_3',
    \end{align*}
    \begin{align*}
        f_1':L_2'K_2'\sim M_2'N_2'\sim S_1'R_1', && f_2':L_2'M_2'\sim K_2'N_2'\sim S_2'R_2', && f_3':N_2'L_2'\sim S_3'R_3',
    \end{align*}
    \begin{align*}
        g':B'A'\sim C'D'\sim P_1'S_1' \sim Q_1'R_1', && h':C'B' \sim D'A' \sim P_2'S_2' \sim Q_2'R_2',
    \end{align*}
    \begin{align*}
        i':A'C'\sim P_3'S_3' \sim Q_3'R_3'.
    \end{align*}
    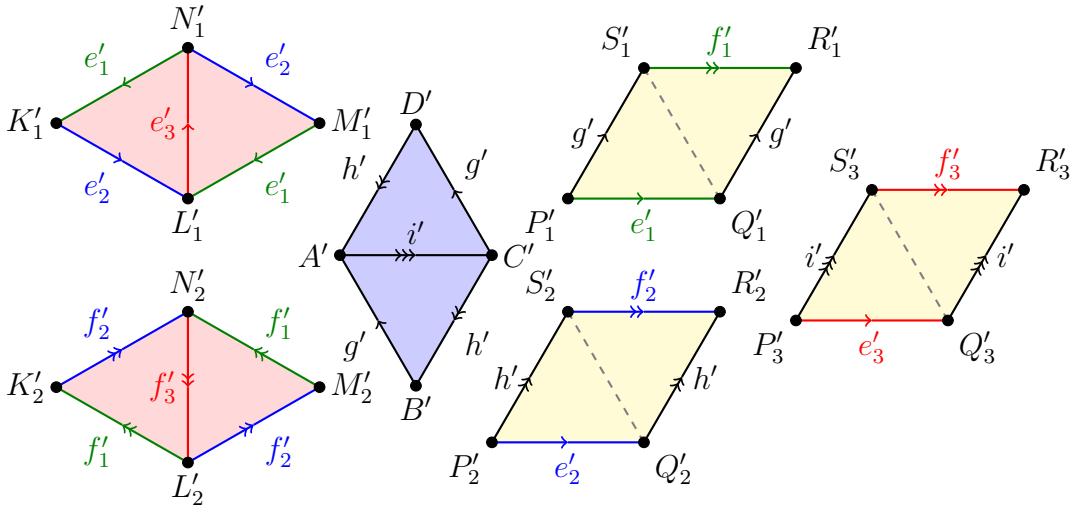
\begin{figure}[!ht]
    \begin{center}
\begin{tikzpicture}
\fill[fill=red!15] (6,0.75)--(7.732,1.75)--(6,2.75);
\fill[fill=red!15] (6,0.75)--(4.268,1.75)--(6,2.75);
\draw[red, thick, ->] (6,0.75) -- (6,1.75);
\draw[red, thick] (6,1.75) -- (6,2.75);
\draw[black!50!green, thick, ->] (7.732,1.75) -- (6.866,1.25);
\draw[black!50!green, thick] (6,0.75) -- (6.866,1.25);
\draw[blue, thick, ->] (6,2.75) -- (6.866,2.25);
\draw[blue, thick] (7.732,1.75) -- (6.866,2.25);
\draw[blue, thick, ->] (4.268,1.75) -- (5.134,1.25);
\draw[blue, thick] (6,0.75) -- (5.134,1.25);
\draw[black!50!green, thick, ->] (6,2.75) -- (5.134,2.25);
\draw[black!50!green, thick] (5.134,2.25) -- (4.268,1.75);
\filldraw[black!50!green] (5.134,2.25) circle (0pt) node[anchor=south east]{$e_1'$};
\filldraw[black!50!green] (6.866,1.25) circle (0pt) node[anchor=north west]{$e_1'$};
\filldraw[blue] (6.866,2.25) circle (0pt) node[anchor=south west]{$e_2'$};
\filldraw[blue] (5.134,1.25) circle (0pt) node[anchor=north east]{$e_2'$};
\filldraw[red] (6,1.75) circle (0pt) node[anchor=east]{$e_3'$};
\filldraw[black] (7.732,1.75) circle (2pt) node[anchor=west]{$M_1'$};
\filldraw[black] (4.268,1.75) circle (2pt) node[anchor=east]{$K_1'$};
\filldraw[black] (6,0.75) circle (2pt) node[anchor=north]{$L_1'$};
\filldraw[black] (6,2.75) circle (2pt) node[anchor=south]{$N_1'$};

\fill[fill=red!15] (6,-0.75)--(7.732,-1.75)--(6,-2.75);
\fill[fill=red!15] (6,-0.75)--(4.268,-1.75)--(6,-2.75);
\draw[red, thick, ->>] (6,-0.75) -- (6,-1.75);
\draw[red, thick] (6,-1.75) -- (6,-2.75);
\draw[black!50!green, thick, ->>] (7.732,-1.75) -- (6.866,-1.25);
\draw[black!50!green, thick] (6,-0.75) -- (6.866,-1.25);
\draw[blue, thick, ->>] (6,-2.75) -- (6.866,-2.25);
\draw[blue, thick] (7.732,-1.75) -- (6.866,-2.25);
\draw[blue, thick, ->>] (4.268,-1.75) -- (5.134,-1.25);
\draw[blue, thick] (6,-0.75) -- (5.134,-1.25);
\draw[black!50!green, thick, ->>] (6,-2.75) -- (5.134,-2.25);
\draw[black!50!green, thick] (5.134,-2.25) -- (4.268,-1.75);
\filldraw[black!50!green] (5.134,-2.25) circle (0pt) node[anchor=north east]{$f_1'$};
\filldraw[black!50!green] (6.866,-1.25) circle (0pt) node[anchor=south west]{$f_1'$};
\filldraw[blue] (6.866,-2.25) circle (0pt) node[anchor=north west]{$f_2'$};
\filldraw[blue] (5.134,-1.25) circle (0pt) node[anchor=south east]{$f_2'$};
\filldraw[red] (6,-1.75) circle (0pt) node[anchor=east]{$f_3'$};
\filldraw[black] (7.732,-1.75) circle (2pt) node[anchor=west]{$M_2'$};
\filldraw[black] (4.268,-1.75) circle (2pt) node[anchor=east]{$K_2'$};
\filldraw[black] (6,-0.75) circle (2pt) node[anchor=south]{$N_2'$};
\filldraw[black] (6,-2.75) circle (2pt) node[anchor=north]{$L_2'$};

\fill[fill=blue!20] (8,0)--(9,1.732)--(10,0)--(9,-1.732);

\draw[black, thick, ->>>] (8,0)--(9,0);
\draw[black, thick] (9,0)--(10,0);
\filldraw[black] (9,0) circle (0pt) node[anchor=south]{$i'$};

\draw[black, thick, ->>] (9,1.732)--(8.5,0.866);
\draw[black, thick] (8,0)--(8.5,0.866);
\draw[black, thick] (9,-1.732)--(9.5,-0.866);
\draw[black, thick, ->>] (10,0)--(9.5,-0.866);
\filldraw[black] (8.5,0.866) circle (0pt) node[anchor=south east]{$h'$};
\filldraw[black] (9.5,-0.866) circle (0pt) node[anchor=north west]{$h'$};

\draw[black, thick] (8,0)--(8.5,-0.866);
\draw[black, thick, ->] (9,-1.732)--(8.5,-0.866);
\draw[black, thick] (9.5,0.866)--(9,1.732);
\draw[black, thick, ->] (10,0)--(9.5,0.866);
\filldraw[black] (8.5,-0.866) circle (0pt) node[anchor=north east]{$g'$};
\filldraw[black] (9.5,0.866) circle (0pt) node[anchor=south west]{$g'$};

\filldraw[black] (10,0) circle (2pt) node[anchor=west]{$C'$};
\filldraw[black] (9,1.732) circle (2pt) node[anchor=south]{$D'$};
\filldraw[black] (8, 0) circle (2pt) node[anchor=east]{$A'$};
\filldraw[black] (9,-1.732) circle (2pt) node[anchor=north]{$B'$};

\fill[fill=yellow!20] (11,0.75)--(13,0.75)--(14,2.482)--(12,2.482);
\draw[black, thick, ->] (11,0.75)--(11.5,1.616);
\draw[black, thick] (11.5,1.616)--(12,2.482);
\filldraw[black] (11.5,1.616) circle (0pt) node[anchor=east]{$g'$};
\draw[black, thick, ->] (13,0.75)--(13.5,1.616);
\draw[black, thick] (13.5,1.616)--(14,2.482);
\filldraw[black] (13.5,1.616) circle (0pt) node[anchor=west]{$g'$};
\draw[black!50!green, thick, ->] (11,0.75)--(12,0.75);
\draw[black!50!green, thick] (12,0.75)--(13,0.75);
\filldraw[black!50!green] (12,0.75) circle (0pt) node[anchor=north]{$e_1'$};
\draw[black!50!green, thick, ->>] (12,2.482)--(13,2.482);
\draw[black!50!green, thick] (13,2.482)--(14,2.482);
\filldraw[black!50!green] (13,2.482) circle (0pt) node[anchor=south]{$f_1'$};
\draw[black!50, thick, dashed] (12,2.482)--(13,0.75);
\filldraw[black] (11,0.75) circle (2pt) node[anchor=north east]{$P_1'$};
\filldraw[black] (13,0.75) circle (2pt) node[anchor=north west]{$Q_1'$};
\filldraw[black] (14,2.482) circle (2pt) node[anchor=south west]{$R_1'$};
\filldraw[black] (12,2.482) circle (2pt) node[anchor=south east]{$S_1'$};

\fill[fill=yellow!20] (11,-0.75)--(13,-0.75)--(12,-2.482)--(10,-2.482);
\draw[black, thick, ->>] (10,-2.482)--(10.5,-1.616);
\draw[black, thick] (10.5,-1.616)--(11,-0.75);
\filldraw[black] (10.5,-1.616) circle (0pt) node[anchor=east]{$h'$};
\draw[black, thick, ->>] (12,-2.482)--(12.5,-1.616);
\draw[black, thick] (12.5,-1.616)--(13,-0.75);
\filldraw[black] (12.5,-1.616) circle (0pt) node[anchor=west]{$h'$};
\draw[blue, thick, ->] (10,-2.482)--(11,-2.482);
\draw[blue, thick] (11,-2.482)--(12,-2.482);
\filldraw[blue] (11,-2.482) circle (0pt) node[anchor=north]{$e_2'$};
\draw[blue, thick, ->>] (11,-0.75)--(12,-0.75);
\draw[blue, thick] (12,-0.75)--(13,-0.75);
\draw[black!50 ,thick, dashed] (11,-0.75)--(12,-2.482);
\filldraw[blue] (12,-0.75) circle (0pt) node[anchor=south]{$f_2'$};
\filldraw[black] (10,-2.482) circle (2pt) node[anchor=north east]{$P_2'$};
\filldraw[black] (12,-2.482) circle (2pt) node[anchor=north west]{$Q_2'$};
\filldraw[black] (13,-0.75) circle (2pt) node[anchor=south west]{$R_2'$};
\filldraw[black] (11,-0.75) circle (2pt) node[anchor=south east]{$S_2'$};

\fill[fill=yellow!20] (14,-0.866)--(16,-0.866)--(17,0.866)--(15,0.866);
\draw[black, thick, ->>>] (14,-0.866)--(14.5,0);
\draw[black, thick] (14.5,0)--(15,0.866);
\filldraw[black] (14.5,0) circle (0pt) node[anchor=east]{$i'$};
\draw[black, thick, ->>>] (16,-0.866)--(16.5,0);
\draw[black, thick] (16.5,0)--(17,0.866);
\filldraw[black] (16.5,0) circle (0pt) node[anchor=west]{$i'$};
\draw[red, thick, ->] (14,-0.866)--(15,-0.866);
\draw[red, thick] (15,-0.866)--(16,-0.866);
\filldraw[red] (15,-0.866) circle (0pt) node[anchor=north]{$e_3'$};
\draw[red, thick, ->>] (15,0.866)--(16,0.866);
\draw[red, thick] (16,0.866)--(17,0.866);
\draw[black!50, thick, dashed] (16,-0.866)--(15,0.866);
\filldraw[red] (16,0.866) circle (0pt) node[anchor=south]{$f_3'$};
\filldraw[black] (14,-0.866) circle (2pt) node[anchor=north east]{$P_3'$};
\filldraw[black] (16,-0.866) circle (2pt) node[anchor=north west]{$Q_3'$};
\filldraw[black] (17,0.866) circle (2pt) node[anchor=south west]{$R_3'$};
\filldraw[black] (15,0.866) circle (2pt) node[anchor=south east]{$S_3'$};

\end{tikzpicture}
\end{center}
    \caption{The polyhedral complex $X_1'$}
    \label{example_1_systolization}
\end{figure}

    Note that the three rhombi $K_1'L_1'M_1'N_1', K_2'L_2'M_2'N_2', A'B'C'D'$ are the same as the corresponding rhombi in the construction of a polyhedral complex $X_1$ and all of the identifications of edges are also the same as in the construction of $X_1$. Thus the whole complex $X_1'$ is constructed by replacing the squares in the construction of $X_1$ by the corresponding rhombi. Such rhombi are homeomorphic to the squares we are replacing, therefore the complexes $X_1$ and $X_1'$ are also homeomorphic. Since each of the rhombi of the complex described above can be obtained by gluing two unit equilateral triangles then $X_1'$ is actually a triangle complex with a single vertex $v'$. Now it is enough to say that $\wtl{X_1'}$ is systolic. By using the fact \ref{2-dim-systolic} and theorems \ref{GLC}, \ref{CH} we can do so by showing that there are no isometric cycles of length less than $2\pi$ inside link $\Lk(v',X_1')$. This link is presented in Figure \ref{link_1}. It is clear that there are no isometric cycles of length less than $2\pi$ in $\Lk(v',X_1')$, therefore $\wtl{X_1'}$ is systolic.
\end{proof}

From the theorem \ref{systolic-biautomatic} we can now conclude that the fundamental group $\pi_1(X_1) \simeq \pi_1(X_1')$ is biautomatic.

\section{Second example}\label{second_example}

Let $X_2$ be a complex presented in the figure \ref{example_2}, that is consisting of a unit hexagon $ABCDEF$, three unit squares $P_1Q_1R_1S_1$, $P_2Q_2R_2S_2$ and $P_3Q_3R_3S_3$ and two unit equilateral triangles $K_1L_1M_1$, $K_2L_2M_2$ glued together in such a way that the following groups of oriented edges are identified (see Figure \ref{example_2}):
$$f_1:AF\sim DC\sim S_1P_1\sim Q_1R_1,$$ 
$$f_2:CB\sim ED\sim S_2P_2\sim Q_2R_2,$$ 
$$f_3:BA\sim FE\sim S_3P_3\sim Q_3R_3,$$ 
$$e_1:P_1Q_1\sim S_1R_1\sim K_1M_1\sim K_2M_2,$$
$$e_2:P_2Q_2\sim S_2R_2\sim M_1L_1\sim L_2K_2,$$
$$e_3:P_3Q_3\sim S_3R_3\sim L_1K_1\sim M_2L_2.$$
Such polyhedral complex has only one vertex $v$, shown in Figure \ref{example_2} as a black dot. Note that since a regular hexagon can be obtained by gluing six equilateral triangles, then $X_2$ can be treated as a triangle-square complex built out of $8$ triangles and $3$ squares.

\begin{figure}[!h]
    \centering
\begin{tikzpicture}
\fill[fill=yellow!20] (0,0)--(2,0)--(2,2)--(0,2);
\fill[fill=yellow!20] (0,3.5)--(2,3.5)--(2,5.5)--(0,5.5);
\fill[fill=yellow!20] (-3.5,1.75)--(-1.5,1.75)--(-1.5,3.75)--(-3.5,3.75);
\draw[blue, thick, ->] (0,0)--(1,0);
\draw[blue, thick] (1,0)--(2,0);
\filldraw[blue] (1,0) circle (0pt) node[anchor=south]{$e_1$};
\draw[blue, thick, ->] (0,2)--(1,2);
\draw[blue, thick] (1,2)--(2,2);
\filldraw[blue] (1,2) circle (0pt) node[anchor=north]{$e_1$};
\draw[black!50!green, thick, ->] (0,2)--(0,1);
\draw[black!50!green, thick] (0,1)--(0,0);
\filldraw[black!50!green] (0,1) circle (0pt) node[anchor=east]{$f_1$};
\draw[black!50!green, thick, ->] (2,0)--(2,1);
\draw[black!50!green, thick] (2,1)--(2,2);
\filldraw[black!50!green] (2,1) circle (0pt) node[anchor=west]{$f_1$};

\draw[blue, thick, ->>] (0,3.5)--(1,3.5);
\draw[blue, thick] (1,3.5)--(2,3.5);
\filldraw[blue] (1,3.5) circle (0pt) node[anchor=south]{$e_2$};
\draw[blue, thick, ->>] (0,5.5)--(1,5.5);
\draw[blue, thick] (1,5.5)--(2,5.5);
\filldraw[blue] (1,5.5) circle (0pt) node[anchor=north]{$e_2$};
\draw[black!50!green, thick, ->>] (0,5.5)--(0,4.5);
\draw[black!50!green, thick] (0,4.5)--(0,3.5);
\filldraw[black!50!green] (0,4.5) circle (0pt) node[anchor=east]{$f_2$};
\draw[black!50!green, thick, ->>] (2,3.5)--(2,4.5);
\draw[black!50!green, thick] (2,4.5)--(2,5.5);
\filldraw[black!50!green] (2,4.5) circle (0pt) node[anchor=west]{$f_2$};

\draw[blue, thick, ->>>] (-3.5,1.75)--(-2.5,1.75);
\draw[blue, thick] (-2.5,1.75)--(-1.5,1.75);
\filldraw[blue] (-2.5,1.75) circle (0pt) node[anchor=south]{$e_3$};
\draw[blue, thick, ->>>] (-3.5,3.75)--(-2.5,3.75);
\draw[blue, thick] (-2.5,3.75)--(-1.5,3.75);
\filldraw[blue] (-2.5,3.75) circle (0pt) node[anchor=north]{$e_3$};
\draw[black!50!green, thick, ->>>] (-3.5,3.75)--(-3.5,2.75);
\draw[black!50!green, thick] (-3.5,2.75)--(-3.5,1.75);
\filldraw[black!50!green] (-3.5,2.75) circle (0pt) node[anchor=east]{$f_3$};
\draw[black!50!green, thick, ->>>] (-1.5,1.75)--(-1.5,2.75);
\draw[black!50!green, thick] (-1.5,3.75)--(-1.5,2.75);
\filldraw[black!50!green] (-1.5,2.75) circle (0pt) node[anchor=west]{$f_3$};

\fill[fill=blue!20] (4.5,1.018)--(6.5,1.018)--(7.5,2.75)--(6.5,4.482)--(4.5,4.482)--(3.5,2.75);
\draw[black!50, thick, dashed] (4.5,1.018)--(6.5,4.482);
\draw[black!50, thick, dashed] (6.5,1.018)--(4.5,4.482);
\draw[black!50, thick, dashed] (3.5,2.75)--(7.5,2.75);
\draw[black!50!green, thick, ->] (4.5,4.482)--(5.5,4.482);
\draw[black!50!green, thick] (5.5,4.482)--(6.5,4.482);
\filldraw[black!50!green] (5.5,4.482) circle (0pt) node[anchor=south]{$f_1$};
\draw[black!50!green, thick, ->>] (6.5,4.482)--(7,3.616);
\draw[black!50!green, thick] (7,3.616)--(7.5,2.75);
\filldraw[black!50!green] (7,3.616) circle (0pt) node[anchor=south west]{$f_2$};
\draw[black!50!green, thick, ->>>] (7.5,2.75)--(7,1.884);
\draw[black!50!green, thick] (7,1.884)--(6.5,1.018);
\filldraw[black!50!green] (7,1.884) circle (0pt) node[anchor=north west]{$f_3$};
\draw[black!50!green, thick, ->] (6.5,1.018)--(5.5,1.018);
\draw[black!50!green, thick] (5.5,1.018)--(4.5,1.018);
\filldraw[black!50!green] (5.5,1.018) circle (0pt) node[anchor=north]{$f_1$};
\draw[black!50!green, thick, ->>>] (4.5,1.018)--(4,1.884);
\draw[black!50!green, thick] (4,1.884)--(3.5,2.75);
\filldraw[black!50!green] (4,1.884) circle (0pt) node[anchor=north east]{$f_3$};
\draw[black!50!green, thick, ->>] (3.5,2.75)--(4,3.616);
\draw[black!50!green, thick] (4,3.616)--(4.5,4.482);
\filldraw[black!50!green] (4,3.616) circle (0pt) node[anchor=south east]{$f_2$};


\fill[fill=red!15] (9,0)--(10.732,1)--(9,2);
\fill[fill=red!15] (9,3.5)--(10.732,4.5)--(9,5.5);
\draw[blue, thick, ->] (9,0)--(9,1);
\draw[blue, thick] (9,2)--(9,1);
\filldraw[blue] (9,1) circle (0pt) node[anchor=east]{$e_1$};
\draw[blue, thick, ->>] (9,2)--(9.866,1.5);
\draw[blue, thick] (9.866,1.5)--(10.732,1);
\filldraw[blue] (9.866,1.5) circle (0pt) node[anchor=south west]{$e_2$};
\draw[blue, thick, ->>>] (10.732,1)--(9.866,0.5);
\draw[blue, thick] (9.866,0.5)--(9,0);
\filldraw[blue] (9.866,0.5) circle (0pt) node[anchor=north west]{$e_3$};

\draw[blue, thick, ->] (9,3.5)--(9,4.5);
\draw[blue, thick] (9,5.5)--(9,4.5);
\filldraw[blue] (9,4.5) circle (0pt) node[anchor=east]{$e_1$};
\draw[blue, thick, ->>>] (9,5.5)--(9.866,5);
\draw[blue, thick] (9.866,5)--(10.732,4.5);
\filldraw[blue] (9.866,5) circle (0pt) node[anchor=south west]{$e_3$};
\draw[blue, thick, ->>] (10.732,4.5)--(9.866,4);
\draw[blue, thick] (9.866,4)--(9,3.5);
\filldraw[blue] (9.866,4) circle (0pt) node[anchor=north west]{$e_2$};


\filldraw[black] (0,0) circle (2pt) node[anchor=north east]{$P_1$};
\filldraw[black] (2,0) circle (2pt) node[anchor=north west]{$Q_1$};
\filldraw[black] (2,2) circle (2pt) node[anchor=south west]{$R_1$};
\filldraw[black] (0,2) circle (2pt) node[anchor=south east]{$S_1$};

\filldraw[black] (0,3.5) circle (2pt) node[anchor=north east]{$P_2$};
\filldraw[black] (2,3.5) circle (2pt) node[anchor=north west]{$Q_2$};
\filldraw[black] (2,5.5) circle (2pt) node[anchor=south west]{$R_2$};
\filldraw[black] (0,5.5) circle (2pt) node[anchor=south east]{$S_2$};

\filldraw[black] (-3.5,1.75) circle (2pt) node[anchor=north east]{$P_3$};
\filldraw[black] (-1.5,1.75) circle (2pt) node[anchor=north west]{$Q_3$};
\filldraw[black] (-1.5,3.75) circle (2pt) node[anchor=south west]{$R_3$};
\filldraw[black] (-3.5,3.75) circle (2pt) node[anchor=south east]{$S_3$};

\filldraw[black] (4.5,1.018) circle (2pt) node[anchor=north east]{$F$};
\filldraw[black] (6.5,1.018) circle (2pt) node[anchor=north west]{$A$};
\filldraw[black] (7.5,2.75) circle (2pt) node[anchor=west]{$B$};
\filldraw[black] (6.5,4.482) circle (2pt) node[anchor=south west]{$C$};
\filldraw[black] (4.5,4.482) circle (2pt) node[anchor=south east]{$D$};
\filldraw[black] (3.5,2.75) circle (2pt) node[anchor=east]{$E$};

\filldraw[black] (9,0) circle (2pt) node[anchor=north]{$K_1$};
\filldraw[black] (10.732,1) circle (2pt) node[anchor=west]{$L_1$};
\filldraw[black] (9,2) circle (2pt) node[anchor=south]{$M_1$};

\filldraw[black] (9,3.5) circle (2pt) node[anchor=north]{$K_2$};
\filldraw[black] (10.732,4.5) circle (2pt) node[anchor=west]{$L_2$};
\filldraw[black] (9,5.5) circle (2pt) node[anchor=south]{$M_2$};

\end{tikzpicture}

\caption{The polyhedral complex $X_2$.}
\label{example_2}
\end{figure}
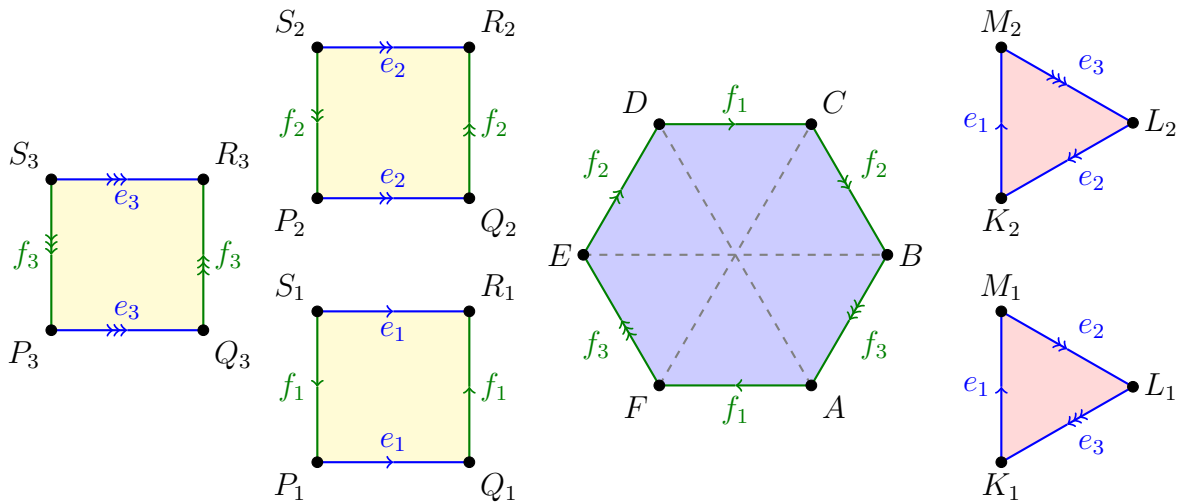

\begin{lm}
    The universal cover $\wtl{X_2}$ is CAT(0). Moreover, if we replace all the squares in construction of $X_2$ with rhombi having the smaller internal angle equal to $\frac{\pi}{3}$, then the universal cover $\wtl{X_2'}$ of polyhedral complex $X_2'$ obtained in such a way fails to be CAT(0).
\end{lm} 

\begin{proof}
From the theorems \ref{GLC} and \ref{CH} we know that it is enough to show that each injective loop in $\Lk(v, X_2)$ has length at least $2\pi$. The link $\Lk(v, X_2)$ is presented in Figure \ref{link_2} where, for the picture readability, top and bottom vertices of a link are identified. It is easy to see that there are no injective loops in $\Lk(v, X_2)$ of length smaller than $2\pi$, thus $X_2$ is locally CAT(0) and $\wtl{X_2}$ is a CAT(0) space.

Again, to check that $\wtl{X_2'}$ is not CAT(0) it is enough to check that $X_2'$ is not locally CAT(0), which can easily be done by analyzing $8$ possible ways in which complex $X_2'$ can be obtained and analyzing links in them. In Figure \ref{link_2} one of the $8$ analogous cases of how the link could look like is analyzed. All the edges in that picture have length of $\frac{\pi}{3}$, so there is a cycle of length $\frac{5\pi}{6}$. Such cycle is outlined in blue in figure \ref{link_2}.

\end{proof}

\begin{figure}[!ht]
    \begin{center}
\begin{tikzpicture}[scale=0.6]
    \draw[red,thick, dashdotted] (0,0)--(3,0)--(5,2)--(2,2)--(0,0);
    \filldraw[black] (0,0) circle (2.75pt) node[anchor=east]{$f_1$};
    \filldraw[black] (3,0) circle (2.75pt) node[anchor=west]{$e_1$};
    \filldraw[black] (2,2) circle (2.75pt) node[anchor=north west]{$\ol{e_1}$};
    \filldraw[black] (5,2) circle (2.75pt) node[anchor=west]{$\ol{f_1}$};

    \draw[red,thick, dashdotted] (0,4)--(3,4)--(5,6)--(2,6)--(0,4);
    \filldraw[black] (0,4) circle (2.75pt) node[anchor=east]{$\ol{f_2}$};
    \filldraw[black] (3,4) circle (2.75pt) node[anchor=west]{$\ol{e_2}$};
    \filldraw[black] (2,6) circle (2.75pt) node[anchor=north west]{$e_2$};
    \filldraw[black] (5,6) circle (2.75pt) node[anchor=west]{$f_2$};

    \draw[red,thick, dashdotted] (0,8)--(3,8)--(5,10)--(2,10)--(0,8);
    \filldraw[black] (0,8) circle (2.75pt) node[anchor=east]{$f_3$};
    \filldraw[black] (3,8) circle (2.75pt) node[anchor=west]{$e_3$};
    \filldraw[black] (2,10) circle (2.75pt) node[anchor=north west]{$\ol{e_3}$};
    \filldraw[black] (5,10) circle (2.75pt) node[anchor=west]{$\ol{f_3}$};

    \draw[red,thick, dashdotted] (0,12)--(3,12)--(5,14)--(2,14)--(0,12);
    \filldraw[black] (0,12) circle (2.75pt) node[anchor=east]{$\ol{f_1}$};
    \filldraw[black] (3,12) circle (2.75pt) node[anchor=west]{$\ol{e_1}$};
    \filldraw[black] (2,14) circle (2.75pt) node[anchor=north west]{$e_1$};
    \filldraw[black] (5,14) circle (2.75pt) node[anchor=west]{$f_1$};

     \draw[black,thick] (3,0)--(3,4)--(3,8)--(3,12);
     \draw[black,thick] (2,2)--(2,6)--(2,10)--(2,14);
     \draw[black,thick] (0,0)--(-1,2)--(0,4)--(-1,6)--(0,8)--(-1,10)--(0,12);
     \draw[black,thick] (5,2)--(6,4)--(5,6)--(6,8)--(5,10)--(6,12)--(5,14);
     \filldraw[black] (-1,2) circle (2.75pt);
     \filldraw[white] (-1,2) circle (1.83pt);
     \filldraw[black] (-1,6) circle (2.75pt);
     \filldraw[white] (-1,6) circle (1.83pt);
     \filldraw[black] (-1,10) circle (2.75pt);
     \filldraw[white] (-1,10) circle (1.83pt);
     \filldraw[black] (6,4) circle (2.75pt);
     \filldraw[white] (6,4) circle (1.83pt);
     \filldraw[black] (6,8) circle (2.75pt);
     \filldraw[white] (6,8) circle (1.83pt);
     \filldraw[black] (6,12) circle (2.75pt);
     \filldraw[white] (6,12) circle (1.83pt);
\end{tikzpicture}\hspace{0.5cm}
\begin{tikzpicture}[scale=0.6]

    \draw[black,thick] (0,0)--(3,0)--(4.5,0.5)--(5,2)--(3.5,2.7)--(2,2)--(0,0);
    \filldraw[black] (0,0) circle (2.75pt) node[anchor=east]{$f_1$};
    \filldraw[black] (3,0) circle (2.75pt) node[anchor=north west]{$e_1$};
    \filldraw[black] (5,2) circle (2.75pt) node[anchor=west]{$\ol{f_1}$};
    \filldraw[black] (4.5,0.5) circle (2.75pt);
    \filldraw[white] (4.5,0.5) circle (1.83pt);
    \filldraw[black] (3.5,2.7) circle (2.75pt);
    \filldraw[white] (3.5,2.7) circle (1.83pt);

    \draw[black,thick] (2,6)--(0.5,5.5)--(0,4)--(1.5,3.3);
    \draw[black,thick] (3,4)--(5,6)--(2,6);
    \filldraw[black] (0,4) circle (2.75pt) node[anchor=east]{$\ol{f_2}$};
    \filldraw[black] (0.5,5.5) circle (2.75pt);
    \filldraw[white] (0.5,5.5) circle (1.83pt);

    \draw[black,thick] (0,8)--(3,8)--(4.5,8.5)--(5,10)--(3.5,10.7)--(2,10)--(0,8);
    \filldraw[black] (0,8) circle (2.75pt) node[anchor=east]{$f_3$};
    \filldraw[black] (2,10) circle (2.75pt) node[anchor=north west]{$\ol{e_3}$};
    \filldraw[black] (5,10) circle (2.75pt) node[anchor=west]{$\ol{f_3}$};
    \filldraw[black] (4.5,8.5) circle (2.75pt);
    \filldraw[white] (4.5,8.5) circle (1.83pt);
    \filldraw[black] (3.5,10.7) circle (2.75pt);
    \filldraw[white] (3.5,10.7) circle (1.83pt);

    \draw[black,thick] (0,12)--(1.5,11.3)--(3,12)--(5,14)--(2,14)--(0.5,13.5)--(0,12);
    \filldraw[black] (0,12) circle (2.75pt) node[anchor=east]{$\ol{f_1}$};
    \filldraw[black] (2,14) circle (2.75pt) node[anchor=north west]{$e_1$};
    \filldraw[black] (5,14) circle (2.75pt) node[anchor=west]{$f_1$};
    \filldraw[black] (1.5,11.3) circle (2.75pt);
    \filldraw[white] (1.5,11.3) circle (1.83pt);
    \filldraw[black] (0.5,13.5) circle (2.75pt);
    \filldraw[white] (0.5,13.5) circle (1.83pt);

     \draw[black,thick] (3,0)--(3,4);
     \draw[black, thick] (3,4)--(3,8)--(3,12);
     \draw[black, thick] (2,2)--(2,6);
     \draw[black,thick] (2,6)--(2,10)--(2,14);
     \draw[black,thick] (0,0)--(-1,2)--(0,4)--(-1,6)--(0,8)--(-1,10)--(0,12);
     \draw[black,thick] (5,2)--(6,4)--(5,6)--(6,8)--(5,10)--(6,12)--(5,14);
     \filldraw[black] (-1,2) circle (2.75pt);
     \filldraw[white] (-1,2) circle (1.83pt);
     \filldraw[black] (-1,6) circle (2.75pt);
     \filldraw[white] (-1,6) circle (1.83pt);
     \filldraw[black] (-1,10) circle (2.75pt);
     \filldraw[white] (-1,10) circle (1.83pt);
     \filldraw[black] (6,4) circle (2.75pt);
     \filldraw[white] (6,4) circle (1.83pt);
     \filldraw[black] (6,8) circle (2.75pt);
     \filldraw[white] (6,8) circle (1.83pt);
     \filldraw[black] (6,12) circle (2.75pt);
     \filldraw[white] (6,12) circle (1.83pt);
     
    \draw[blue!50, line width = 2pt] (2,2)--(2,6)--(5,6)--(3,4)--(3,5);
    \draw[black,thick] (1.5,3.3)--(3,4);
    \filldraw[black] (1.5,3.3) circle (2.75pt);
    \filldraw[white] (1.5,3.3) circle (1.83pt);
    \filldraw[blue!50] (3,4) circle (4.75pt) node[anchor=west]{$\ol{e_2}$};
    \filldraw[black] (3,4) circle (2.75pt);
    \filldraw[blue!50] (2,6) circle (4.75pt) node[anchor=north west]{$e_2$};
    \filldraw[black] (2,6) circle (2.75pt);
    \filldraw[blue!50] (2,2) circle (4.75pt) node[anchor=north west]{$\ol{e_1}$};
    \filldraw[black] (2,2) circle (2.75pt);
    \filldraw[blue!50] (5,6) circle (4.75pt) node[anchor=west]{$f_2$};
    \filldraw[black] (5,6) circle (2.75pt);
    \draw[black, thick] (2,2)--(2,6)--(5,6)--(3,4)--(3,5);

    \draw[blue!50, line width = 2pt] (3,5)--(3,12);
    \filldraw[blue!50] (3,8) circle (4.75pt) node[anchor=north west]{$e_3$};
    \filldraw[black] (3,8) circle (2.75pt);
    \filldraw[blue!50] (3,12) circle (4.75pt) node[anchor=west]{$\ol{e_1}$};
    \filldraw[black] (3,12) circle (2.75pt);
    \draw[black, thick] (3,5)--(3,12);
\end{tikzpicture}
\end{center}
    \caption{Left: link of $v$ in $X_2$, Right: a possible link of $v'$ in $X_2'$.}
    \label{link_2}
\end{figure}
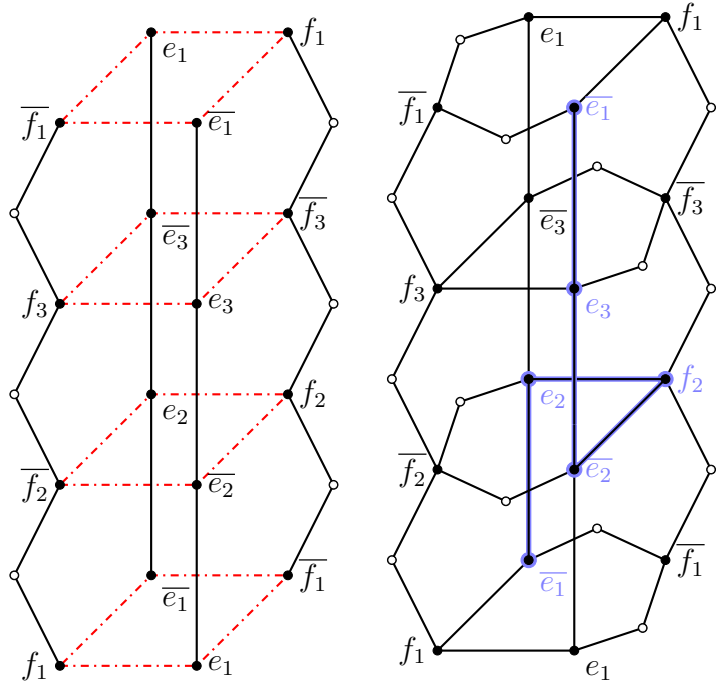

\begin{lm}\label{F_embeds_into_X2}
    The radial flat $F$ embeds in the universal cover $\wtl{X_2}$.
\end{lm}

\begin{proof}
    Similarly to proof of lemma \ref{F_embeds_into_X1}, it is enough to check that there is a cellular map off triangle-square complexes $\varphi\colon F\to X_2$ locally isometric on each cell such that for each vertex $w$ of $F$ the induced map $\Lk(w,F) \to \Lk(\varphi(w),X_2)$ is injective. Again, each cell of $X_2$ is uniquely determined by oriented edges on its boundary, so it is enough to define such map on each edge of $F$. Such map is presented in the figure \ref{flat_2_1} and its formal description is presented below. 

    \begin{figure}[!ht]
    \centering
    \includegraphics[trim=1.1cm 3.7cm 1.1cm 4.2cm, clip, width=\textwidth]{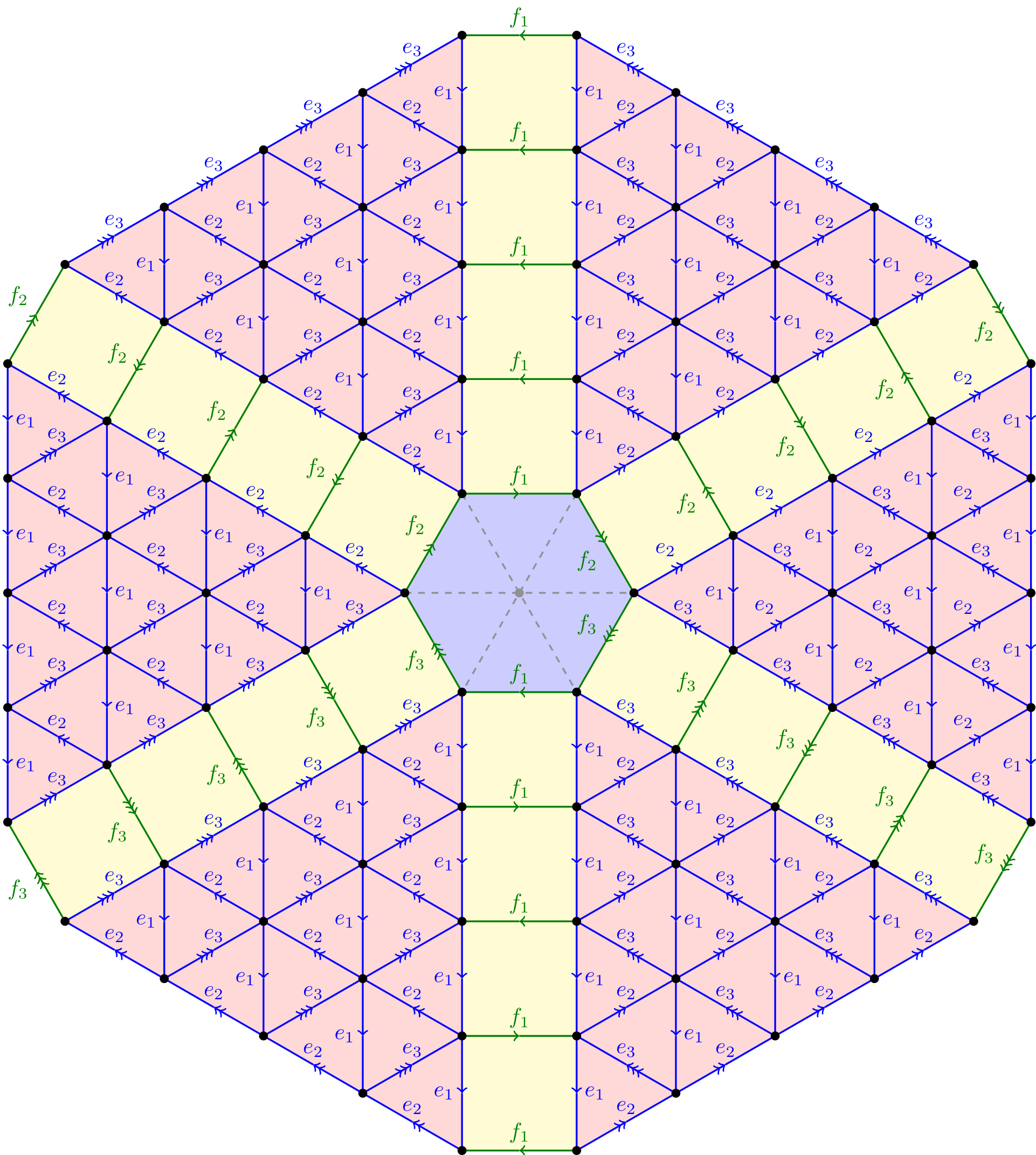}
    \caption{A piece of a radial flat $F$ that embeds into $\wtl{X}$}
    \label{flat_2_1}
\end{figure}

    Let us set the Cartesian coordinates on flat $F$ in the same way as in proof of lemma \ref{F_embeds_into_X1}, that is set $o$ as its center and set one edge of the central hexagon to be parallel to the line $y=0$. We will now define $\varphi$ for the edges in the unbounded triangle regions in the right halfplane bounded by the line $x=0$. Let every oriented edge of form $((x,y),(x,y-1))$ be mapped to $e_1$, every edge of form $((x,y),(x+\sqrt{3}/2,y+1/2))$ be mapped to $e_2$ and every edge of form $((x,y),(x-\sqrt{3}/2,y+1/2))$ be mapped to $e_3$. Moreover we can now define $\varphi$ for the edges in the unbounded triangle regions in a left halfplane bounded by the line $x=0$ in a following way: if an oriented edge $e'$ is a reflection of $e$ by the line $x=0$, then $\varphi(e)=\varphi(e')$.  Now let us define inductively the map $\varphi$ for all the edges in the upper half plane parallel to the line $y=0$. Let the edge $((-1/2,\sqrt{3}/2),(1/2,\sqrt{3}/2))$ be mapped to $f_1$, and whether the edge $((x,y),(x',y))$ is mapped to $f_1$, then the edge $((x',y+1),(x,y+1))$ is mapped to $f_1$ (note the change of direction). Now we can define the map $\varphi$ for the rest of the edges by saying that if the oriented edge $e$ was mapped to $f_1$, then the edges obtained from $e$ by rotating by $\pm \frac{\pi}{3}$ around $o$ are mapped to $f_2$, the edges obtained from $e$ by rotating by $\pm \frac{2\pi}{3}$ around $o$ are mapped to $f_3$ and the edges obtained from $e$ by rotating by $\pi$ around $o$ are mapped to $f_1$.
    
    It is clear that for every vertex $w$ of $F$ the induced map of the links is injective. Therefore $F$ embeds in the universal cover $\wtl{X_2}$.
\end{proof}

\begin{lm}
    Each of the flats $F_{2n-1}$ for $n\in \NN_+$ embeds in the universal cover $\wtl{X_2}$.
\end{lm}

\begin{proof}
    Similarly to the proofs of the lemmas \ref{F_embeds_into_X1}, \ref{Fn_embeds_into_X1}, \ref{F_embeds_into_X2} it is enough to check that there is a cellular map of triangle-square complexes $\varphi_{2n-1}\colon F_{2n-1}\to X_2$ locally isometric on each cell such that for each vertex $w$ of $F_{2n-1}$ the induced map $\Lk(w,F_{2n-1}) \to \Lk(\varphi_{2n-1}(w),X_2)$ is injective. In figure \ref{flat_2_2} we present such map for $n=2$. We will now construct such map directly in a similar fashion to the proof of lemma \ref{Fn_embeds_into_X1}. Again, we will use the fact that each cell in the complex $X_2$ is defined uniquely by its edges which allows us to only define them map $\varphi_{2n-1}$ on edges of the flat $F_{2n-1}$.

    \begin{figure}[!ht]
    \centering
    \includegraphics[trim=8cm 50cm 8cm 50cm, clip, width=\textwidth]{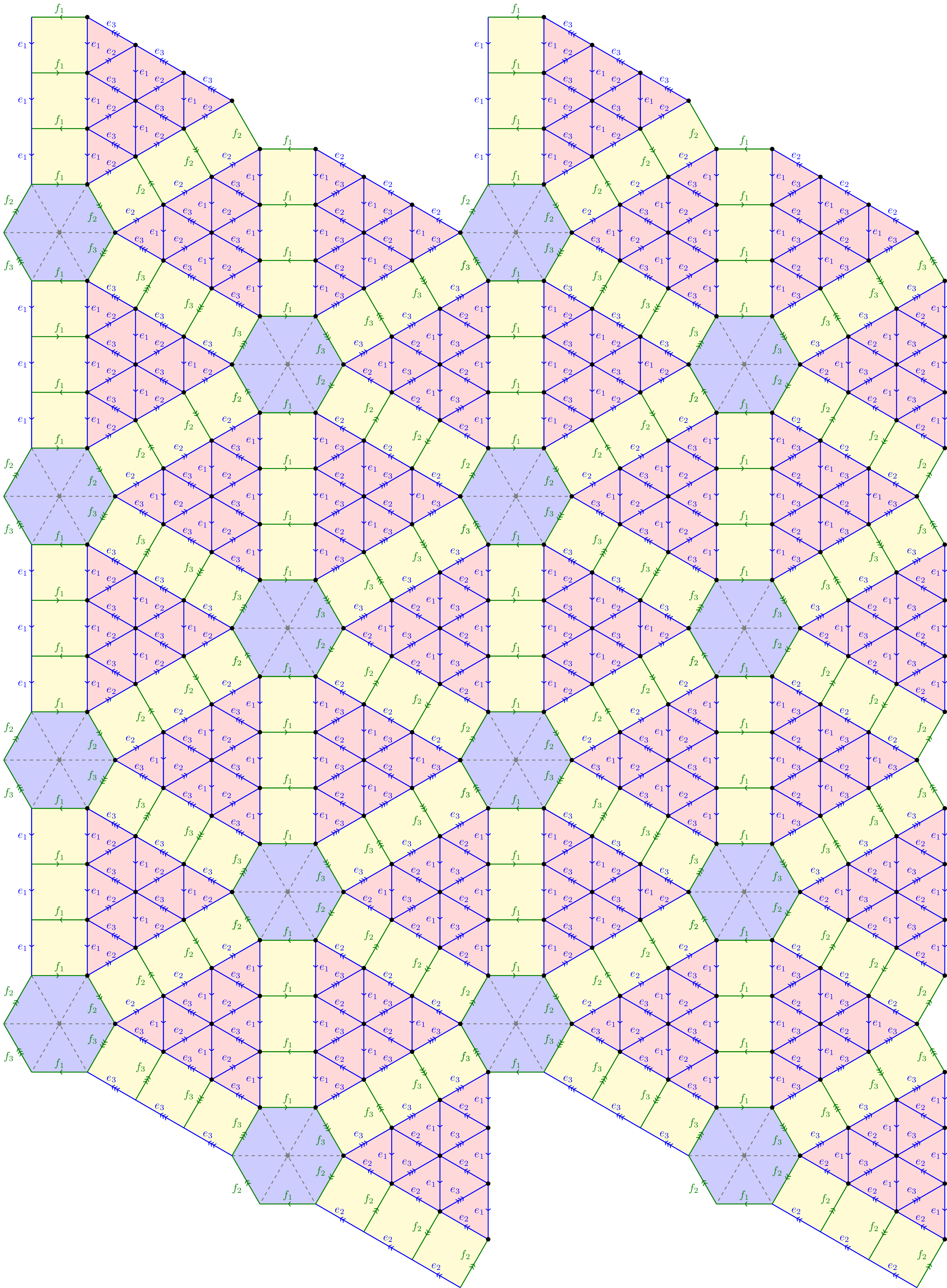}
    \caption{A part of thoroughly crumpled flat $F_3$ that embeds into $\wtl{X_2}$.}
    \label{flat_2_2}
\end{figure}

    Let the polygon $D_{2n-1}$ be defined as in the proof of the lemma \ref{Fn_embeds_into_X1}. We will now set the Cartesian coordinates on $D_{2n-1}$ in such a way that the point $(0,0)$ is the center of the hexagon used in the construction and $((1/2,-\sqrt{3}/2),(1,0)), ((1,0),(1/2,\sqrt{3}/2)), ((1/2,\sqrt{3}/2),(-1/2,\sqrt{3}/2))$ are edges of the hexagon to which the rectangles are glued. Now let $E_{2n-1}$ be a polygon defined as a union of $D_{2n-1}$ and $D_{2n-1}$ translated by vector $(3/2+\sqrt{3}n-\sqrt{3}/2, -n+1/2-\sqrt{3}/2)$. It is easy to check that we can tile a plane by translation of $E_{2n-1}$ of form $(A(2\sqrt{3}n-\sqrt{3}+3),B(\sqrt{3}+2n-1))$ for $A,B\in \ZZ$. Such plane treated as a triangle square complex is in fact isomorphic to the plane $F_{2n-1}$. We will now define a map $\psi_{2n-1}$ on $E_n$.
    
    Firstly we will define the ancillary automorphism $\sigma_{23}$ on the complex $X_2$. Since each cell of $X_2$ is uniquely defined by its oriented edges, we can define the map $\sigma_{23}$ on such edges in a following way:
    \begin{align*}
        \sigma_{23}(e_1) = e_1, && \sigma_{23}(e_2) = e_3, && \sigma_{23}(e_3)=e_2, &&\sigma_{23}(f_1) = f_1, && \sigma_{23}(f_2) = f_3, && \sigma_{23}(f_3) = f_2.
    \end{align*}
    It is easy to check that a map defined in such a way extends to a cellular automorphism.
    Now we can define the map $\psi_{2n-1}$ on $E_{2n-1}$ in a following way: we define such map on the non\=/translated copy of $D_{2n-1}$ as a restriction of the map $\varphi$ from the proof of the lemma \ref{F_embeds_into_X2} and on the translated copy of $D_{2n-1}$ is defined as a composition $\sigma_{23}\circ\varphi \circ \tau_{(-3/2-\sqrt{3}n+\sqrt{3}/2, n-1/2+\sqrt{3}/2)}$ of $\sigma_{23}$, the map $\varphi$ from the proof of the lemma \ref{F_embeds_into_X2} and a translation by vector $(-3/2-\sqrt{3}n+\sqrt{3}/2, n-1/2+\sqrt{3}/2)$. Now let us consider a group $\mathcal{T}$ generated by translations by vectors $(2\sqrt{3}n-\sqrt{3}+3,0)$ and $(0,\sqrt{3}+2n-1)$. The map $\varphi_{2n-1}$ is defined as:

    $$\varphi_{2n-1}(x) = \psi_{2n-1}(t^{-1}(x))\text{, where }t\in \mathcal{T}\text{ is such that }x\in t(E_{2n-1}).$$
    
    In Figure \ref{flat_2_2} each edge is labeled with edge onto which it is mapped by $\varphi_{2n+1}$ which uniquely determines which cell of $X_2$ the given cell of $F_{2n+1}$ is mapped to. It is clear that for every vertex $w$ of $F$ the induced map of the links is injective. Therefore $F_{2n+1}$ embeds in the universal cover $\wtl{X_2}$.
\end{proof}

\section{Final remarks}\label{final remarks}

The counterexamples presented above imply that the Gersten-Short geodesics in general cannot be used to construct a regular path system satisfying fellow traveler property. It is still possible however that if we restrict ourselves to compact complexes not violating conjectures \ref{conj:1} and \ref{conj:2} the Gersten-Short geodesics will create a regular path system ensuring that fundamental group of $X$ will be biautomatic. This allow us to state the following conjecture:

\begin{conj}
    Let $X$ be a compact locally CAT(0) triangle-square complex such that no intrinsically aperiodic flat embed in $\wtl{X}$ and there are only finitely many distinct thoroughly crumpled flats that embed in $\wtl{X}$. Then there is a vertex $v$ of $X$ such that Gersten-Short geodesics starting and ending in points of orbit $\pi_1(X)\cdot v$ form a regular path system satisfying the generalized fellow traveler property. In particular $\pi_1(X)$ is biautomatic.
\end{conj}

Furthermore, even if the Gersten-Short geodesics do not yield satisfying result regarding biautomaticity of groups acting geometrically and cellularly on CAT(0) triangle-square complexes it is entirely possible that all such groups are biautomatic for some other, yet unknown reason. As an example of possible behaviour described above the first of the presented counterexamples turns out to be biautomatic since the triangle-square complexe on which it acts is systolizable. We conjecture that the fundamental group of the complex $X_2$ is also biautomatic.

\begin{conj}\label{conj_X2}
    The fundamental group $\pi_1(X_2)$ is biautomatic.
\end{conj}

Moreover, the conjecture from the original paper of R. Levitt and J. McCammond remains open.

\begin{conj}[conjecture $\mathbf{1.2}$ from \cite{triangles_squares}]\label{1.2_from_triangles_squares}
    The fundamental group of a compact locally CAT(0) triangle-square complex is biautomatic.
\end{conj}

A proof of the conjecture \ref{conj_X2} may be important, since it potentially could provide a new method for studying fundamental groups of locally CAT(0) triangle-square complexes, potentially enabling the progress in proving conjecture \ref{1.2_from_triangles_squares}. On the other hand, a counterexample to the conjecture \ref{1.2_from_triangles_squares} would also be interesting as it would give an example of  non-biautomatic $2$-dimensional CAT(0) group, answering one of the long-standing open questions in the theory of biautomatic groups.

The search of possible biautomatic structures on presented counterexamples can also possibly be executed using an appropriate computer software. One such method is an algorithm \textsc{kbmag2}, described in a paper \cite{automatic_search} and further developed by D. Holt with coauthors \cite{holt2000kbmag2}. Described algorithms are however sometimes heavy and may not be suitable for work with larger complexes. This raises an final new open problem regarding the CAT(0) triangle-square complexes.

\begin{oprob}
    Is the locally CAT(0) triangle-square complex $X_1$ the complex violating conjectures \ref{conj:1} and \ref{conj:2} with the smallest number of cells? Is the locally CAT(0) triangle-square complex $X_2$ the complex violating conjectures \ref{conj:1} and \ref{conj:2} with the smallest number of cells that is not systolizable? Describe more examples of locally CAT(0) triangle-square complexes violating conjectures \ref{conj:1} and \ref{conj:2} with relatively small number of cells.
\end{oprob}


\bibliographystyle{amsalpha}
\bibliography{mybibliography}

\vspace*{\fill}
{\large\textsc{Instytut Matematyczny, Uniwersytet Wrocławski, plac Grunwaldzki 2, 50-384 Wrocław,
Poland\medskip\\and\medskip\\Department of Mathematical Sciences, University of Copenhagen, Universitetspark 5,
2100 Copenhagen, Denmark}\medskip\\
\textit{Email address: }\href{mailto:mk@math.ku.dk}{mk@math.ku.dk}\medskip\\
\textsc{Keywords and phrases:} biautomaticity, triangle-square complexes, CAT(0), nonpositive curvature}

\end{document}